\documentclass[11pt]{amsart}
\usepackage[leqno]{amsmath}
\usepackage{amsthm}
\usepackage{amsfonts, tikz-cd, enumerate,url}
\usetikzlibrary{babel}
\usepackage{amssymb}
\usepackage[all]{xy}
\usepackage{mathtools}
\usepackage{stmaryrd}
\usepackage[colorinlistoftodos]{todonotes}

\DeclareFontFamily{OT1}{rsfs}{}
\DeclareFontShape{OT1}{rsfs}{n}{it}{<-> rsfs10}{}
\DeclareMathAlphabet{\mathscr}{OT1}{rsfs}{n}{it}

\usepackage{xspace}
\usepackage{epsfig,epic,eepic,latexsym,color}
\usepackage[all]{xy}
\usepackage{comment} 
\usepackage[modulo]{lineno}
\usepackage{aliascnt} 
\usepackage[colorlinks=true,linkcolor=blue,citecolor=blue,urlcolor=blue,citebordercolor={0 0 1},urlbordercolor={0 0 1},linkbordercolor={0 0 1}]{hyperref} 
\usepackage[nameinlink]{cleveref}
\crefname{diagram}{Diagram}{Diagram}

\usepackage{latexsym}

\theoremstyle{plain}
\newtheorem{theorem}{Theorem}[section]
\newtheorem{corollary}[theorem]{Corollary}
\newtheorem{lemma}[theorem]{Lemma}
\newtheorem{proposition}[theorem]{Proposition}

\newtheorem{conjecture}[theorem]{Conjecture}

\theoremstyle{definition}
\newtheorem{definition}[theorem]{Definition}
\newtheorem{example}[theorem]{Example}

\theoremstyle{remark}
\newtheorem{remark}[theorem]{Remark}

\numberwithin{equation}{section}

\newcommand{\co}{\colon}

\newcommand{\spec}{\operatorname{Spec}}

\newcommand{\APerf}{\operatorname{APerf}}
\newcommand{\modd}{\operatorname{Mod}}
\newcommand{\Coh}{\operatorname{Coh}}
\newcommand{\qcoh}{\operatorname{qcoh}}
\newcommand{\coh}{\operatorname{coh}}
\newcommand{\Hom}{\operatorname{Hom}}
\newcommand{\Ext}{\operatorname{Ext}}

\newcommand{\GL}{\mathrm{GL}}
\newcommand{\SL}{\mathrm{SL}}

\newcommand{\rgamma}{\mathsf{R}\Gamma}
\newcommand{\cX}{\mathcal{X}}
\newcommand{\cF}{\mathcal{F}}
\newcommand{\cO}{\mathcal{O}}
\newcommand{\cI}{\mathcal{I}}
\newcommand{\cJ}{\mathcal{J}}
\newcommand{\cY}{\mathcal{Y}}
\newcommand{\cG}{\mathcal{G}}
\newcommand{\cU}{\mathcal{U}}
\newcommand{\cZ}{\mathcal{Z}}
\newcommand{\ev}{\operatorname{ev}}
\newcommand{\cE}{\mathcal{E}}
\newcommand{\cC}{\mathcal{C}}

\newcommand{\cM}{\mathcal{M}}

\newcommand{\tensor}{\otimes}
\newcommand{\red}{{\operatorname{red}}}

\newcommand{\oh}{\mathcal{O}}
\newcommand{\bG}{\mathbf{G}}
\newcommand{\Spec}{\operatorname{Spec}}
\renewcommand{\hat}{\widehat}
\newcommand{\limit}{\varprojlim}

\newcommand{\fm}{\mathfrak{m}}
\newcommand{\into}{\hookrightarrow}

\newcommand{\fG}{\mathfrak{G}}
\newcommand{\iso}{\stackrel{\sim}{\to}}
\newcommand{\bA}{\mathbf{A}}
\newcommand{\bZ}{\mathbf{Z}}
\newcommand{\Cent}{\operatorname{Cent}}
\newcommand{\Lie}{\operatorname{Lie}}

\newcommand{\bX}{\mathbb{X}}

\newcommand{\cT}{\mathcal{T}}

\def\gitq{/\hspace{-0.1cm}/}
\newcommand{\id}{\operatorname{id}}

\newcommand{\xto}{\xrightarrow}
\newcommand{\Stab}{\operatorname{Stab}}

\newcommand{\holim}[1]{\mathrm{holim}_{{#1}}}
\newcommand{\LDERF}{\mathsf{L}}
\newcommand{\RDERF}{\mathsf{R}}

\newcommand{\spref}[1]{\href{https://stacks.math.columbia.edu/tag/{#1}}{#1}}
\newcommand{\lisset}{\textrm{lisse,\'etale}}
\usepackage{ulem}
\renewcommand{\emph}{\textit} 

\title{Coherently Complete Algebraic Stacks in Positive Characteristic}
\author{Jarod Alper}
\author{Jack Hall}
\author{David Benjamin Lim}

\address{The University of Washington, Department of Mathematics, Box 354350, Seattle, WA 98195-4350, USA}
\email{jarod@uw.edu}

\address{School of Mathematics \& Statistics\\The University of Melbourne\\Parkville,
  VIC, 3010\\Australia}
\email{jack.hall@unimelb.edu.au}

\email{benlim@alumni.stanford.edu}
\date{\today}

\begin{document}
\maketitle

\begin{abstract}
  With the long-term goal of proving local structure theorems of
  algebraic stacks in positive characteristic near points with
  reductive (but possibly non-linearly reductive) stabilizer, we
  conjecture that quotient stacks of the form $[\Spec A/G]$, with $G$
  reductive and $A^G$ complete local, are coherently complete along
  the unique closed point.  We establish this conjecture in two
  interesting cases: (1) $A^G$ is artinian and (2) $G$ acts trivially
  on $\Spec A$.  We also establish coherent completeness results for
  graded unipotent group actions.  In order to establish these
  results, we prove a number of foundational statements concerning
  cohomological and completeness properties of algebraic
  stacks---including on how these properties ascend and descend along
  morphisms.
\end{abstract}
\setcounter{tocdepth}{1}
\tableofcontents

\section{Introduction}
Let $X$ be a proper scheme over a complete noetherian local ring
$(R, \fm_R)$. There are three important theorems that govern the
formal geometry of $X$, analogous to Grauert's finiteness theorem and
Serre's GAGA theorems in the setting of complex analytic geometry:
\begin{enumerate}[(1)]\label{intro:properties}
\item \emph{Finiteness of Cohomology}\label{intro:properties:finite}:
  For any coherent sheaf $F$ on $X$, the cohomology group
  $\mathrm{H}^i(X,F)$ is a finitely generated $R$-module for all $i$.
\item \emph{Formal Functions}\label{intro:properties:formal}: Let
  $X_0 := X \otimes_R R/\fm_R$ to be the central fiber, $X_n$ the
  $n$th nilpotent thickening of $X_0$, and $\hat{X}$ the formal
  completion of $X$ along $X_0$.  Then for any coherent sheaf $F$ on
  $X$, there are natural isomorphisms:
    \[
      \mathrm{H}^i(X,F) \simeq \mathrm{H}^i(\hat{X}, \hat{F}) \simeq
      \limit_n \mathrm{H}^i(X_n, F|_{X_n}).
    \]
  \item \emph{Formal GAGA}\label{intro:properties:GAGA}: There are equivalences of categories
    \[
      \Coh(X) \iso \Coh(\widehat{X}) \iso \limit_n \Coh(X_n).
    \]
  \end{enumerate}
  These properties and their extension to proper algebraic stacks
  \cite{ref:conrad,ref:olsson} are powerful tools in modern algebraic
  geometry. For example, they can be used to answer lifting questions
  about schemes to characteristic zero, or to prove the proper base
  change theorem in \'{e}tale cohomology. In the context of moduli
  theory, Artin's Criteria \cite{artin-versal} implies that the
  effectivization of formal objects (e.g.\ curves, sheaves) is a
  necessary condition for a moduli stack to be algebraic. In practice,
  effectivization reduces to checking some incarnation of formal GAGA.

This paper investigates to what extent these three properties hold for (non-separated) algebraic stacks.  Of primary interest are quotient stacks of the form $[\Spec A/G]$, where the invariant ring $A^G$ is a complete local ring.  To formulate our results, we introduce the following three definitions mirroring the three properties above:

\begin{enumerate}[(1)]
  \item  A noetherian algebraic stack $\cX$ over a noetherian ring $R$ is \emph{cohomologically proper} if for every coherent sheaf $\cF$ on $\cX$, the cohomology group $\mathrm{H}^i(\cX, \cF)$ is a finitely generated $R$-module for all $i$.
\end{enumerate}
Let $\cX$ be a noetherian algebraic stack and $\cX_0 \subset \cX$ a
closed substack.  Let $\widehat{\cX}$ be the completion of $\cX$ along
$\cX_0$; i.e. the ringed site
$(\cX_{\text{lis-\'{e}t}}, \widehat{\cO}_{\cX,\cX_0})$, where
$\widehat{\cO}_{\cX,\cX_0} \coloneqq \varprojlim
\cO_{\cX}/\mathcal{I}^{n+1}$ (see \Cref{def:completions}).

  \begin{enumerate}[(1)]\setcounter{enumi}{1}
  \item We say that the pair $(\cX, \cX_0)$ satisfies \emph{formal functions} if for every coherent sheaf $\cF$ on $\cX$, the natural map
  $$\mathrm{H}^i(\cX, \cF) \to \mathrm{H}^i(\widehat{\cX}, \widehat{\cF})$$
  is an isomorphism for all $i$.
  \item We say that the pair $(\cX,\cX_0)$ is \textit{coherently complete} if
  the natural functor
  \[ \Coh(\cX) \to \Coh(\widehat{\cX})\]
  is an equivalence of categories.
\end{enumerate} 
See \Cref{S:cp,S:ff-cc} for a further discussion of these definitions as well as examples.

\subsection{Main conjecture and results}

\begin{conjecture} \label{conj:main}
  Let $G$ be a smooth geometrically reductive group scheme over a complete noetherian local ring $R$.  Let $X = \Spec A$ be an affine scheme of finite type over $R$ with an action of $G$ such that $A^G=R$. Let $x \in [X/G]$ be the unique closed point and $\cG_x$ be its residual gerbe. Then $[X/G]$ is cohomologically proper over $A^G$, and the pair $([X/G], \cG_x)$ is coherently complete and satisfies formal functions.
\end{conjecture}

This conjecture differs from formal GAGA in two ways. First, when $G$ is not finite, the quotient stack $[\Spec A /G]$ is not separated, let alone proper, over $\Spec A^G$.  Second, in formal GAGA, the completion is taken with respect to the central fiber (i.e. the pullback of the maximal ideal from the base), whereas the completion in  \Cref{thm:main} is taken with respect to the \textit{smaller} closed substack defined by the residual gerbe at the unique closed point.  For instance, under the standard action of $\SL_2$ on  $\bA^2 = \spec k[x,y]$, the origin is the unique closed orbit and the invariant ring is $k[x,y]^{\SL_2} = k$.  Since the central fiber of $[\bA^2 / \SL_2] \to \Spec k$ is everything, the coherent completeness of $[\bA^2 / \SL_2]$ along the central fiber has no mathematical content.  On the other hand, the coherent completeness of $[\bA^2 / \SL_2]$ along the origin $B \SL_2$ is a non-trivial statement.

This conjecture was established in \cite[Thm.~1.6]{ref:ahr2} in the case that $G$ is linearly reductive.  While in characteristic $0$ the notions of linearly reductivity and reductivity agree, in positive characteristic linear reductivity is a very strong notion: a smooth affine group scheme is linearly reductive if it is an extension of a torus by a finite group whose order is prime to the characteristic \cite{MR142667}.  For example $\GL_n$ is reductive in any characteristic but linearly reductive only in characteristic 0.   

\begin{theorem} \label{thm:main}
  \Cref{conj:main} holds in the following cases
  \begin{enumerate}[(1)]
      \item \label{thm:main1} 
        $R$ is an artinian local ring, or 
      \item \label{thm:main2} 
        $G$ acts trivially on $X$, i.e. $A=R$ and $[X/G] = BG$.
  \end{enumerate}
\end{theorem}

We expect that these two cases will assist in establishing
\Cref{conj:main} in general.  Indeed, \eqref{thm:main1} reduces the
statement that $[X/G]$ is coherently complete and satisfies formal
functions with respect to the residual gerbe $\cG_x$ to seemingly
simpler statement that $[X/G]$ satisfies these properties with respect
to the central fiber of $[\Spec A/G] \to \Spec A^G$. Our methods essentially
reduce \Cref{conj:main} to the `generically toric' situation
(\Cref{R:generically-linearly-reductive}).

In fact, we establish a stronger result than \eqref{thm:main2}: if $G$ is a geometrically reductive group scheme over an $I$-adically complete noetherian ring $R$, then $BG$ is cohomologically proper over $R$, and the pair $(BG, B G_{R/I})$ is coherently complete and satisfies formal functions (\Cref{thm:reductive}).
When $G$ is the base change of a reductive group defined over a field, this had been proved in \cite[Prop.~4.3.4]{ref:hlp}.

It turns out that our results hold for a wider class than reductive group schemes:  they hold for graded unipotent groups, which are precisely the type of groups that arise in non-reductive GIT \cite{bdhk}.

\begin{theorem} \label{thm:positive-graded}
  Let $G \to \Spec R$ be a smooth affine group scheme over a complete local noetherian ring $R$ with residue field $k$. Let $G_u \subset G$ be its unipotent radical with reductive quotient $G_r = G/G_u$.  
  Let $\lambda \co \bG_m \to G$ be a one-parameter subgroup such that $\lambda$ is central in $G_r$ and acts positively on $G_u$.  Then
  \begin{enumerate}[(1)]
    \item \label{thm:positive-graded1}
      $BG$ is cohomologically proper over $R$, and $(BG, BG_k)$ is coherently complete and satisfies formal functions; and
    \item \label{thm:positive-graded2} if $G$ acts on an affine scheme
      $X = \Spec A$ of finite type over $R$ such that $A^G = R$,
      $[X/G_r]$ satisfies \Cref{conj:main} (e.g., $G_r$ is linearly
      reductive or $R$ is artinian) and $\lambda$ acts semipositively
      on $A$, then $[X/G]$ is cohomologically proper, and there is a
      unique closed point $x \in [X/G]$ such that $([X/G], \cG_x)$ is
      coherently complete and satisfies formal functions.
  \end{enumerate}
\end{theorem}

For example, if $B \subset G$ is a Borel subgroup of a split reductive group scheme over $R$, then $B = P_{\lambda}$ for a regular one-parameter subgroup $\lambda \co \bG_m \to G$.  The reductive quotient $B_r=T$ of $B$ is a maximal torus (thus linearly reductive) such that $\lambda$ is central.  The theorem implies that $BB$ is cohomologically proper and that $(BB, BB_k)$ is coherently complete, and moreover that the same holds for a quotient stack $[\Spec A/G]$ if $A^B=R$ and $\lambda$ acts semipositively on $A$.  

On the other hand, these three cohomological properties do not hold for every algebraic group $G$, e.g. unipotent groups (see \Cref{ex:unipotent}).

\subsection{Methods and other results}
Our strategy to establish \Cref{thm:main} is via descent.  We prove that if $f \co \cX \to \cY$ is a universally submersive morphism and $\cY_0 \subset \cY$ is a closed substack such that $\cX$ and the higher base changes $\cX \times_{\cY} \cdots \times_{\cY} \cX$ satisfy the three cohomological properties---cohomologically properness, formal functions, and coherent completeness---along the preimage of $\cY_0$, then $\cY$ satisfies these properties along $\cY_0$ (\Cref{thm:descent}).  The two most important cases are either when $f$ is proper and surjective, in which case one only needs to check that $\cX$ has the properties along the preimage of $\cY_0$, and when $f$ is faithfully flat.

To apply the descent result, we construct suitable covers of $[\Spec A/G]$.  For the case of $BG_R$ for a reductive group $G$ over a complete local noetherian ring $R$ (i.e. \Cref{thm:main}\eqref{thm:main2}), we first reduce to the case that $G$ is split reductive.  We consider subgroups $T \subset B \subset G$ where $T$ is a maximal torus and $B$ is a Borel subgroup, and consider the induced morphisms
$$BT \to BB \to BG.$$
Since $G/B$ is projective, the map $BB \to BG$ is proper and surjective, and hence \Cref{thm:descent} reduces the theorem to $BB$.  On the other hand, $BT \to BB$ is faithfully flat and affine.  Since $T$ is a torus and hence linearly reductive, we know the cohomological properties for $BT$.  Moreover, the higher base changes of $BT \to BB$ are identified with quotient stacks $[U^{\times (n-1)} / T]$ where $U$ is the quotient $B/T$ of the right action, and $T$ acts on $U$ on the left.  The $T$-invariants of the action on the $(n-1)$-fold product $U^{\times (n-1)}$ are identified with $U$, which implies that $[U^{\times (n-1)} / T] \to \Spec R$ is a good moduli space and thus satisfies the three cohomological properties.  \Cref{thm:descent} then implies that $BB$ satisfies the desired properties.

The construction of the covers in the case of $[\Spec A/G]$ with $A^G$ artinian (i.e., \Cref{thm:main}\eqref{thm:main1}) is more involved.  The general strategy is to find a one-parameter subgroup $\lambda \co \bG_m \to G$ and consider the
pair of morphisms
$$[X^+_{\lambda}/C_{\lambda}] \to [X^+_{\lambda}/P_{\lambda}] \to [X/G],$$
where $X^+_{\lambda} = {\rm Hom}^{\bG_m}(\bA^1, X)$ is the `attractor' locus  parameterizing points of $X$ that have a limit under $\bG_m$-action induced by $\lambda$ (see \S \ref{sec:fixed-attractor}), and 
where $C_{\lambda}$ and $P_{\lambda}$ are the centralizer and parabolic of $\lambda$ (see \S\ref{subsec:centralizer-parabolic}).  For any $\lambda$, the map $[X^+_{\lambda}/P_{\lambda}] \to [X/G]$ is proper since $X^+_{\lambda} \subset X$ is a closed subscheme and $G/P_{\lambda}$ is projective, while the map $[X^+_{\lambda}/C_{\lambda}] \to [X^+_{\lambda}/P_{\lambda}]$ is faithfully flat and affine.  A version of the Hilbert--Mumford criterion (\Cref{P:generic-destabilization}) implies that there is one-parameter subgroup $\lambda$ destabilizing at least one point in every $G$-orbit, which gives that $G \cdot X^+_{\lambda} = X$ or in other words that $[X^+_{\lambda}/P_{\lambda}] \to [X/G]$ is surjective.  The descent result (\Cref{thm:descent}) therefore reduces to the claim to $[X^+_{\lambda}/P_{\lambda}]$.  However, $[X^+_{\lambda}/C_{\lambda}]$ (in addition to the higher base changes of $[X^+_{\lambda}/C_{\lambda}] \to [X^+_{\lambda}/P_{\lambda}]$) do not clearly satisfy the cohomological properties as $C_{\lambda}$, while reductive, is not linearly reductive.  However, we establish a refined version of the Hilbert--Mumford criterion (\Cref{prop:refined-destabilization}) that yields a one-parameter subgroup $\lambda$ such that every $G$-orbit contains a point $x$ and such that $x_0 := \lim_{t \to 0} \lambda(t) \cdot x$ exists and lies in the unique closed $G$-orbit, and moreover such that the induced one-parameter subgroup is \emph{regular} in the stabilizer $G_{x_0}$.  This implies that the connected component of $G_{x_0} \cap C_{\lambda}$ is linearly reductive, which suffices to show that $[X^+_{\lambda}/C_{\lambda}]$ and the higher base changes satisfy the desired cohomological properties. 

We also establish a number of other foundational results, such as:
\begin{itemize}
  \item every \emph{strongly cohomologically proper} morphism (see \Cref{D:strong-cp}) satisfies formal functions (\Cref{T:strong-cp-implies-ff}); and
  \item coherent completeness ascends along proper and representable morphisms:  if $f\co \cX \to \cY$ is a proper representable morphism and $(\cY, \cY_0)$ is coherently complete, then $(\cX, f^{-1}(\cY_0))$ is coherently complete (\Cref{thm:properness-ascent}).
\end{itemize}

\subsection{Motivation and applications}
The motivation for this paper is to extend the local structure theorem of \cite[Thm.~1.1]{ref:ahr} to apply more widely in positive characteristic.  Specifically, \cite[Thm.~1.1]{ref:ahr} implies that a quasi-separated algebraic stack of finite type over an algebraically closed field with affine stabilizers is \'etale locally a quotient stack of the form $[\Spec A/\GL_n]$ near a point $x$ with linearly reductive stabilizer $G_x$; various generalizations of this were established in \cite{ref:ahr2} and \cite{ref:ahhlr}.  The long-term goal of this project is to prove an analogous result near points with reductive stabilizer. 
Such a local structure theorem would be a powerful foundational result in stack theory with wide-ranging applications. 
Most notably, it would immediately extend the existence theorems for moduli spaces proved in \cite[Thm.~A]{ref:ahlh} to positive characteristic.  Namely, it would imply that in any characteristic, an algebraic stack $\cX$ of finite type with affine diagonal admits a separated adequate moduli space if and only $\cX$ is $\Theta$-complete and $\mathsf{S}$-complete.

Coherent completeness plays an essential role in the proof of \cite[Thm.~1.1]{ref:ahr}.  The main idea of the proof in the case that $x \in \cX$ is a smooth point is to consider the quotient $\cT = [T_{\cX,x}/G_x]$ of the Zariski tangent space, and the base change 
$$\hat{\cT} := \cT \times_{T_{\cX,x}\gitq G_x} \Spec \hat{\oh}_{T_{\cX,x}\gitq G_x, 0}$$
along the completion of the GIT quotient at the image of the origin.   Since $\hat{\cT}$ is coherently complete along $BG_x$ by the linearly reductive case of \Cref{conj:main} (i.e. \cite[Thm.~1.6]{ref:ahr2}), Tannaka duality \cite{ref:HR-tannaka} implies that 
$$\Hom(\hat{\cT}, \cX) \cong \varprojlim \Hom(\cT_n, \cX).$$
Using again that $G_x$ is linearly reductive, deformation theory yields isomorphisms $\cT_n \cong \cX_n$ of the $n$th nilpotent thickening of $0 \in \cT$ and $x \in \cX$.  The identification above implies that the maps $\cT_n \iso \cX_n \into \cX$ extend to a map $\hat{\cT} \to \cX$, which can then be approximated to produce the desired \'etale neighborhoods.  In the reductive case, \Cref{conj:main} would yield the coherent completeness of $\hat{\cT}$ and be a big step in the direction of a local structure theorem.  However, even with \Cref{conj:main} resolved, deformation theory would remain an additional obstruction to establish a local structure theorem in the reductive case.

\subsection{Comparison to the literature}
The paper \cite{ref:hlp} discusses related cohomological properties of stacks and establishes a number of related results.  The authors introduce the notion of formal properness (see \Cref{rmk:formal-properness}) and establish for instance that cohomologically projective morphisms \cite[Def.~4.1.3]{ref:hlp} are formally proper \cite[Thm.~4.2.1]{ref:hlp}, which is a variant of \Cref{thm:properness-ascent} (see \Cref{rmk:hlp-properness-ascent}).  They also prove \Cref{thm:main}\eqref{thm:main2} in the case that $G$ is the base change of a reductive group over a field \cite[Prop.~4.3.4]{ref:hlp}. 

One major difference between our paper and \cite{ref:hlp} is that we are concerned with establishing coherent completeness of $[\Spec A/G]$ along the residual gerbe of the unique closed point. In contrast, \cite{ref:hlp} focuses on coherent completeness results of stacks such as $[\Spec A/G]$ along the central fiber of a morphism $[\Spec A/G] \to \Spec A^G$.  In this sense, our results are substantially stronger than \cite{ref:hlp} and are necessary for the intended applications to local structure theorems.  Another feature of our paper is that we utilize only classical methods and avoid derived algebraic geometry.

\subsection{Acknowledgements} We would like to thank Dmitry Kubrak,
Zev Rosengarten, Sean Cotner, Bogdan Zavyalov, Brian Conrad, Jochen
Heinloth, Dan Halpern-Leistner, and Elden Elmanto for very helpful
conversations.

\subsection{Conventions}
We follow the conventions of \cite{ref:stacks}. In particular, if
$\cX$ is an algebraic stack, the lisse-\'etale site $\cX_{\lisset}$
denotes the category with objects the morphisms
$U \xrightarrow{p} \cX$, where $U$ is a scheme and $p$ is smooth;
coverings of $U \to \cX$ are given by \'etale coverings of $U$
\cite[Tag \spref{0787}]{ref:stacks}.

\section{Cohomological properness} \label{S:cp} For a morphism $f \colon X \to Y$ of
noetherian schemes (even algebraic spaces or algebraic stacks with
finite diagonal), it is known that $f$ is proper
if and only if the higher pushforwards $\RDERF^i f_*$ send coherent
sheaves to coherent sheaves for all $i\geq 0$ \cite{182902}.  For
algebraic stacks with infinite stabilizers, there does not exist such
a comparatively simple characterization. It was noted, however, in
\cite[Def.~2.4.1]{ref:hlp} that the higher pushforwards of a morphism
preserving coherent sheaves was a useful condition and satisfied for
various interesting types of algebraic stacks. In this vein, we make
the following definition.
\begin{definition}\label{D:cp}
  \label{def:cp} Let $f \co \cX \to \cY$ be a finite type morphism 
  of noetherian algebraic stacks. We say that $f$ is:
  \begin{enumerate}
  \item \emph{cohomologically proper} if for every coherent sheaf
    $\cF$ on $\cX$, $\RDERF f_\ast \cF \in D^+_{\coh}(\cY)$; and
  \item \emph{universally cohomologically proper} if for every
    morphism $\cY' \to \cY$ of noetherian algebraic stacks, the base
    change $\cX \times_\cY \cY' \to \cY'$ is cohomologically proper.
  \end{enumerate}
\end{definition}
\begin{remark} 
  Universal cohomological properness is equivalent to the
  \emph{coherent pushforward} property introduced in
  \cite[Def.~2.4.1]{ref:hlp}.
\end{remark}
\begin{remark}\label{R:check-cp-on-affines}
  In \Cref{def:cp}, by descent, it suffices to check cohomological
  properness on a smooth cover of the target. In particular, universal
  cohomological properness can be verified by only base changing to
  affine schemes.
\end{remark}
\begin{remark} \label{rmk:cp-equivalence}
  If $f \colon \cX \to \cY$ is cohomologically proper and $\cF \in D^+_{\coh}(\cX)$, then  $\RDERF f_* \cF  \in D^+_{\coh}(\cY)$.
  This follows from the convergent hypercohomology spectral sequence:
  \[ \RDERF^i f_\ast \mathscr{H}^j(\cF) \Rightarrow
    \mathscr{H}^{i+j}(\RDERF f_\ast \cF).  \]
\end{remark}
\begin{remark}\label{R:ucp-uc}
  If $f \colon \cX \to \cY$ is universally cohomologically proper,
  then it is universally closed (e.g., \cite{182902} and \cite[Prop.\
  2.4.5]{ref:hlp}). The basic idea is to apply to the valuative
  criterion for universal closedness \cite[Tag
  \spref{0CLV}]{ref:stacks}, so we are reduced to the situation where
  $\cY= \Spec D$, where $(D,\mathfrak{m})$ is a DVR with fraction
  field $K$ such that $f_K \colon \cX_K \to \Spec K$ has a section
  $s$. We must show that, potentially after extending $D$, that $s$
  extends to a section over $\Spec D$. If
  $p\in f^{-1}(\mathfrak{m}) \neq \emptyset$, then we can find a DVR
  $D'$ over $D$ such that $\Spec D' \to \cX$ realizes the
  specialization from the image of $s$ to $p$. In particular, we must
  show that $f^{-1}(\mathfrak{m}) = \emptyset$ cannot occur if $f$ is
  cohomologically proper. If $f^{-1}(\mathfrak{m}) = \emptyset$, then
  $f$ factors through $j \colon \Spec K \to \Spec R$. It follows that
  $\mathrm{H}^0(\cX,\cO_{\cX})$ is a non-zero finite dimensional
  $K$-vector space, so it cannot be a finitely generated $R$-module,
  which contradicts the cohomological properness of $f$. This proves
  the claim.
\end{remark}
The following two examples will be key in this article. 
\begin{example}\label{E:proper-is-cp}
  Let $f\colon \cX \to \cY$ be a proper morphism of noetherian
  algebraic stacks. Then $f$ is universally cohomologically
  proper. Since properness is stable under arbitrary base change, it
  suffices to prove that proper morphisms are cohomologically proper,
  which follows from \cite{ref:olsson}.
\end{example}
\begin{example}\label{E:fund-is-cp}
  Let $R$ be a noetherian ring. Let $A$ be a finitely generated
  $R$-algebra with an action of a reductive group scheme $G$ over
  $\Spec R$. Then $\pi\colon [\Spec A/G] \to \Spec (A^G)$ is
  universally cohomologically proper. Indeed, 
  by
  \Cref{R:check-cp-on-affines}, it suffices to check that if $B$ is a
  noetherian $A^G$-algebra, then
  $\pi_B \colon [\spec (A \otimes_{A^G} B)/G_B] \to \spec (B)$ is
  cohomologically proper. But $A \otimes_{A^G} B$ is a finitely
  generated $B$-algebra and the result follows from \cite[Thm.\
  10.5]{MR3525842}.
\end{example}

\section{Formal functions and coherent completeness}\label{S:ff-cc}

Throughout we use the following notation:  if $\cX$ is a noetherian algebraic stack and $\cX_0 \subset \cX$ is a closed substack defined by a coherent sheaf of ideals $\mathcal{I} \subset \cO_\cX$, then we denote by $\cX_n \subset \cX$  the closed substack defined by $\cI^{n+1}$.  If $\cF$ is a quasi-coherent sheaf on $\cX$, we denote by $\cF_n$ its pullback to $\cX_n$.

\begin{definition}
  \label{def:completions}
  Let $\cX$ be a noetherian algebraic stack and $\cX_0 \subset \cX$ be a closed substack defined by a coherent sheaf of ideals $\mathcal{I} \subset \cO_\cX$. The \emph{completion of the pair $(\cX, \cX_0)$} is the ringed site
  $\widehat{\cX} \coloneqq (\cX_{\text{lis-\'{e}t}}, \widehat{\cO}_{\cX,\cX_0})$, where the sheaf of rings is defined by the limit
  \[ 
    \widehat{\cO}_{\cX,\cX_0} \coloneqq 
      \varprojlim \cO_{\cX}/\mathcal{I}^{n+1}
  \]
  in the category of lisse-\'{e}tale modules.
\end{definition}

There is a canonical morphism of ringed sites
$c_{\cX,\cX_0}\colon \widehat{\cX} \to \cX$, which is flat
\cite[Lem.~3.3]{ref:gerdzb} and whose pullback preserves coherence; we
will often suppress the subscript by writing
$c_{\cX} \co \widehat{\cX} \to \cX$ or simply
$c\co \widehat{\cX} \to \cX$ if there is little possibility for
confusion. For a coherent $\cO_{\cX}$-module $\cF$, we denote by
$\hat{\cF}$ the pullback of $\cF$ to $\hat{\cX}$. There is a natural
identification $\hat{\cF} \cong \limit \cF/\mathcal{I}^{n+1}\cF$.
There is also an exact equivalence of abelian categories
$\Coh(\hat{\cX}) \simeq \limit \Coh(\hat{\cX}_n)$ \cite[Thm.\
2.3]{ref:conrad}.

\begin{definition}
  \label{def:ff-cc}
  Let $\cX$ be a noetherian algebraic stack and $\cX_0 \subset \cX$ be a closed substack. 
  \begin{enumerate}[(1)]
  \item $(\cX,\cX_0)$ satisfies \textit{formal functions} if for any coherent sheaf $\cF$ on $\cX$, the natural map:
      \[
        \mathrm{H}^i(\cX,\cF) \to \mathrm{H}^i(\hat{\cX},\hat{\cF}) 
      \]
    is an isomorphism for all $i\geq 0$;
    \item $(\cX,\cX_0)$ is \textit{coherently complete} if the functor:
      \[ 
      \Coh(\cX) \to \Coh(\widehat{\cX})\simeq \limit \Coh(\hat{\cX}_n)
      \]
    is an equivalence.
  \item $(\cX,\cX_0)$ satisfies \emph{derived formal functions} if the functor:
    \[
     c^* \colon D_{\coh}(\cX) \to D_{\coh}(\hat{\cX})
    \]
    is fully faithful.
  \item $(\cX,\cX_0)$ is \emph{derived coherently complete} if the functor:
    \[
      c^* \colon D_{\coh}(\cX) \to D_{\coh}(\hat{\cX})
    \]
    is an equivalence. 
  \end{enumerate}
\end{definition}

Examples are littered throughout the next few sections, with several
in \Cref{sec:permanence}.

\begin{remark}\label{R:formal-fns}
  The traditional formulation of ``formal functions'' in the
  literature is that
  $\mathrm{H}^i(\cX,\cF) \simeq \varprojlim_n \mathrm{H}^i(\cX_n,\cF_n)$ for all
  $i\geq 0$ and coherent sheaves $\cF$ on $\cX$. We will refer to this
  as \emph{Zariski formal functions}. If $\cF \in \Coh(\cX)$, then
  $\hat{\cF} = \holim{n} \cF_n$ in $D(\cX)$ \cite[Tag
  \spref{0A0K}]{ref:stacks}. In particular,
  $\RDERF \Gamma(\hat{\cX},\hat{\cF}) \simeq \holim{n} \RDERF
  \Gamma(\cX_n,\cF_n)$. The Milnor exact sequence \cite[Tag
  \spref{07KZ}]{ref:stacks} implies that for each $i$ there is a short
  exact sequence:
  \begin{equation}
    \xymatrix{0 \ar[r] & \varprojlim^1_n \mathrm{H}^{i-1}(\cX_n,\cF_n) \ar[r] &
      \mathrm{H}^i(\hat{\cX},\hat{\cF}) \ar[r] & \varprojlim_n \mathrm{H}^i(\cX_n,\cF_n)
      \ar[r] & 0.}\label{eq:milnor-sequence}
  \end{equation}
  It follows that formal functions implies that
  $\mathrm{H}^i(\cX,\cF) \twoheadrightarrow \varprojlim_n
  \mathrm{H}^i(\cX_n,\cF_n)$ for all $i$ and $\cF$. Conversely,
  Zariski formal functions implies that
  $\mathrm{H}^i(\cX,\cF) \subseteq
  \mathrm{H}^i(\hat{\cX},\hat{\cF})$. Hence, these conditions are not
  obviously equivalent in the generality that we are working. They are
  when $i=0$ or when $\cX$ has affine diagonal, however. This follows from an
  identical argument to that provided in
  \cite[Cor.~V.2.20]{ref:knutson} for separated algebraic
  spaces.
\end{remark}
\begin{remark}[Formal properness] \label{rmk:formal-properness} In
  \cite[Defn.~1.1.3]{ref:hlp}, the property of \emph{formal
    properness} is introduced, which is related to the above
  definition of derived coherent completeness.  A pair $(\cX, \cX_0)$
  consisting of an algebraic stack $\cX$ and a cocompact closed
  substack $\cX_0$ is called \emph{complete} if
  $\APerf(\cX) \iso \APerf(\hat{\cX})$ is an equivalence.  A morphism
  $f \co \cX \to \cY$ of algebraic stacks is \emph{formally proper} if
  for every complete pair $(\cY, \cY'_0)$, the base change
  $\cX \times_{\cY} \cY'$ is complete with respect to
  $\cX \times_{\cY} \cY'_0$.

  If $\cX$ is a noetherian algebraic stack over a complete noetherian
  local ring $\Spec R$ and $x \in \cX$ is a closed point, then the
  conditions of $(\cX, \cG_x)$ being (derived) coherently complete and
  the morphism $\cX \to \Spec R$ being formally proper are
  incomparable.  On one hand, coherent completeness concerns the
  closed substack $\cG_x$ rather than the central fiber
  $\cX \otimes_R R/\fm_R$ while formal properness includes a
  completeness property with respect to all base changes.
\end{remark}

If the functor $\Coh(\cX) \to \Coh(\hat{\cX})$ is fully faithful (e.g.
$(\cX,\cX_0)$ is coherently complete), then for any coherent sheaf
$\cF$ on $\cX$, the natural map
\[
  \mathrm{H}^0(\cX, \cF) \to \mathrm{H}^0(\hat{\cX}, \hat{\cF})
\]
is an isomorphism. Additional hypotheses are needed to imply formal
functions, i.e. that the comparison maps on higher cohomology are
isomorphisms.  On the other hand, if
$D_{\coh}(\cX) \to D_{\coh}(\hat{\cX})$ is fully faithful (i.e.
$(\cX, \cX_0)$ satisfies derived formal functions), then
\[
  \mathrm{H}^i(\cX, \cF) \to \mathrm{H}^i(\hat{\cX}, \hat{\cF}) 
\]
is an isomorphism for all $i$ (i.e. $(\cX, \cX_0)$ satisfies formal
functions).  The proposition below implies in fact that formal
functions is equivalent to derived formal functions, and that derived
coherent completeness is equivalent to coherent completeness and
formal functions.

Let $\cX$ be a noetherian algebraic stack and $\cX_0 \subset \cX$ a
closed substack. The morphism of ringed sites
$c \colon \hat{\cX} \to \cX$ induces an adjoint pair on unbounded
derived categories:
\[
  c^*_{\cX,\cX_0} \colon D_{\qcoh}(\cX) \leftrightarrows D(\hat{\cX}) \colon c_{\cX,\cX_0,{\qcoh},*}.
\]
The functor $c_{\cX,\cX_0,\qcoh,*}$ coincides with the composition of the
forgetful functor \[{c_{\cX,\cX_0,*} \colon D(\hat{\cX}) \to D(\cX)}\] and the
\emph{quasi-coherator}, which is the right adjoint to the inclusion
$D_{\qcoh}(\cX) \to D(\cX)$. See \cite[\S4]{ref:hall} for more
details. We have the following simple proposition. 
\begin{proposition}\label{P:underived-derived-comparison}
  Let $\cX$ be a noetherian algebraic stack and $\cX_0 \subset \cX$ a
  closed substack.
 \begin{enumerate}[{\rm(1)}]
 \item \label{PI:underived-derived-comparison:ff}$(\cX,\cX_0)$
   satisfies formal functions if and only if it satisfies derived
   formal functions. Moreover, if $\cG \in D_{\Coh}(\cX)$, then
   $\cG \to c_{\qcoh,*}c^*\cG$ is an isomorphism.
  \item \label{PI:underived-derived-comparison:cc} 
    If $(\cX,\cX_0)$ satisfies formal functions, then it is
    coherently complete if and only if it is derived coherently
    complete. In this case,
    \begin{enumerate}[{\rm(a)}]
    \item\label{PI:underived-derived-comparison:cc:t-exact}
      $c_{{\qcoh},*}$ is $t$-exact on $D_{\coh}(\hat{\cX})$; and
    \item \label{PI:underived-derived-comparison:cc:qcoh} if
      $\mathfrak{F} \in D_{\coh}(\hat{\cX})$, then
      $c_{{\qcoh},*}\mathfrak{F} \in D_{\coh}(\cX)$ and
      $c^*c_{{\qcoh},*}\mathfrak{F} \simeq \mathfrak{F}$.
    \end{enumerate}
  \end{enumerate}
\end{proposition}
\begin{proof}
  Assume that $(\cX,\cX_0)$ satisfies derived formal functions. Let $\cF
 \in \Coh(\cX)$; then
  \[
    \mathrm{H}^i(\cX,\cF) = \Hom_{\oh_{\cX}}(\oh_{\cX}, \cF[i]) \simeq
    \Hom_{\oh_{\cX}}(\oh_{\hat{\cX}},\hat{\cF}[i]) = \mathrm{H}^i(\hat{\cX},\hat{\cF}).
  \]
  That is, $(\cX,\cX_0)$ satisfies formal functions. Conversely, we
  first note that if $(\cX,\cX_0)$ satisfies formal functions and
  $\mathcal{Q} \in D^+_{\coh}(\cX)$, then
  $\RDERF \Gamma(\cX,\mathcal{Q}) \simeq \RDERF \Gamma(\hat{\cX},c^*\mathcal{Q})$. Indeed, by
  the hypercohomology spectral sequence, we have a diagram
  \[ \begin{tikzcd} \mathrm{H}^i(\cX,\mathscr{H}^j(\mathcal{Q})) \ar{d}{(a)}  \ar[Rightarrow]{r} & \mathrm{H}^{i+j}(\cX, \mathcal{Q}) \ar{d}{(b)}\\
  \mathrm{H}^i(\hat{\cX},\mathscr{H}^j(c^*\mathcal{Q}))  \ar[Rightarrow]{r} & \mathrm{H}^{i+j}(\widehat{\cX}, c^*\mathcal{Q}) \end{tikzcd}\]
where (a) is a morphism of spectral sequences that is an isomorphism by formal functions. Hence (b) is an isomorphism as well. Now
  let $\mathcal{M} \in D^{-}_{\coh}(\cX)$ and $\mathcal{N} \in D^{+}_{\coh}(\cX)$. Then $\RDERF\mathscr{H}om_{\oh_{\cX}}(\mathcal{M},\mathcal{N}) \in D^{+}_{\coh}(\cX)$ and so we have
  \begin{align*}
    \RDERF\Hom_{\oh_{\cX}}(\mathcal{M},\mathcal{N}) &\simeq \RDERF\Gamma(\cX,\RDERF\mathscr{H}om_{\oh_{\cX}}(M,N))\\
                                &\simeq \RDERF \Gamma(\hat{\cX},c^*\RDERF\mathscr{H}om_{\oh_{\cX}}(\mathcal{M},\mathcal{N}))  && (\mbox{formal functions})\\
                                &\simeq \RDERF \Gamma(\hat{\cX},\RDERF\mathscr{H}om_{\oh_{\hat{\cX}}}(c^*\mathcal{M},c^*\mathcal{N}))  && (\mbox{$c$ is flat and $\mathcal{M} \in D^{-}_{\coh}(\cX)$})\\
                                &\simeq \RDERF\Hom_{\oh_{\hat{\cX}}}(c^*\mathcal{M},c^*\mathcal{N})\\
    &\simeq \RDERF \Hom_{\oh_{\cX}}(\mathcal{M},c_{{\qcoh},*}c^*\mathcal{N}).
  \end{align*}
  Let $\cG \in D_{\coh}(\cX)$. To prove that formal functions implies
  derived formal functions, it suffices to prove that the morphism
  $\cG \to c_{{\qcoh},*}c^*\cG$ is an isomorphism in $D(\cX)$. By Lemma
  \ref{L:coherent-gens-der} and the above, $\mathcal{\cG} \to c_{{\qcoh},*}c^*\cG$ is
  an isomorphism whenever $\cG \in D^+_{\coh}(\cX)$. In general, Lemma
  \ref{L:trunc-holim-formal} provides
  \begin{align*}
    c_{{\qcoh},*}c^*\cG &\simeq c_{{\qcoh},*}(\holim{n}\tau^{\geq -n}c^*\cG) \simeq \holim{n}c_{{\qcoh},*}(\tau^{\geq -n}c^*\cG)\\
      &\simeq \holim{n}c_{{\qcoh},*}c^*(\tau^{\geq -n}\cG) \simeq \holim{n} \tau^{\geq -n}\cG \simeq \cG.
  \end{align*}
  This proves \eqref{PI:underived-derived-comparison:ff}. For
  \eqref{PI:underived-derived-comparison:cc}, by
  \eqref{PI:underived-derived-comparison:ff} it suffices to prove that
  the functor is essentially surjective. Equivalently, that the
  natural map $c^*c_{{\qcoh},*}\mathfrak{F} \to \mathfrak{F}$ is an
  isomorphism in $D_{\coh}(\hat{\cX})$. A standard induction argument
  on the length of the complex together with the derived formal
  functions assumption and coherent completeness, shows that for every
  $\mathfrak{F} \in D^b_{\coh}(\hat{\cX})$ there is a unique
  $\cF \in D^b_{\coh}(\cX)$ such that $c^*\cF \simeq \mathfrak{F}$. If
  $\mathcal{P} \in \Coh(\cX)$, 
  it follows from derived formal functions
  that the map $\cF \to c_{{\qcoh},*}\mathfrak{F}$ induces isomorphisms:
  \begin{align*}
    \RDERF \Hom_{\oh_{\cX}}(\mathcal{P},\cF) &\simeq \RDERF
                                   \Hom_{\oh_{\hat{\cX}}}(c^*\mathcal{P},c^*\cF) \simeq \RDERF
                                   \Hom_{\oh_{\hat{\cX}}}(c^*\mathcal{P},\mathfrak{F})\\
                                 &\simeq \RDERF
                                   \Hom_{\oh_{{\cX}}}(\mathcal{P},c_{{\qcoh},*}\mathfrak{F}).
  \end{align*}
  It follows from Lemma \ref{L:coherent-gens-der} that the map
  $\cF \to c_{{\qcoh},*}\mathfrak{F}$ is an isomorphism, so
  $c^*c_{{\qcoh},*}\mathfrak{F} \to \mathfrak{F}$ is an isomorphism. Now
  let $\mathfrak{F} \in D^{-}_{\coh}(\hat{\cX})$. Let
  $\cF_n = c_{{\qcoh},*}(\tau^{\geq -n}\mathfrak{F}) \in D^b_{\coh}(\cX)$
  and set $\cF = \holim{n} \cF_n$. Then \cite[Tag
  \spref{0D6T}]{ref:stacks} (also see Lemma
  \ref{L:trunc-holim-formal}) shows that
  $\mathscr{H}^q(\cF) \simeq \mathscr{H}^q(\cF_n) \in \Coh(\cX)$ whenever
  $q\geq -n$. Hence, $\cF \in D^{-}_{\coh}(\cX)$. Lemma
  \ref{L:trunc-holim-formal} gives a morphism:
  \begin{align*}
    c^*\cF &= c^*(\holim{n} \cF_n) \to \holim{n} c^*\cF_n \simeq \holim{n} \tau^{\geq -n}\mathfrak{F} \simeq \mathfrak{F}.
  \end{align*}
  Let $q\in \bZ$; then the above map induces isomorphisms:
  \[
    \mathscr{H}^q(c^*\cF) \simeq c^*\mathscr{H}^q(\cF) \simeq
    c^*\mathscr{H}^q(\cF_{-q}) \simeq \mathscr{H}^q(c^*\cF_{-q}) \simeq
    \mathscr{H}^q(\tau^{\geq q}\mathfrak{F}) \simeq
    \mathscr{H}^q(\mathfrak{F}).
  \]
  Hence, $c^*F \simeq \mathfrak{F}$. In particular, arguing as above
  using derived formal functions and Lemma \ref{L:coherent-gens-der},
  we see that $c^*c_{{\qcoh},*}\mathfrak{F} \to \mathfrak{F}$ is an
  isomorphism whenever $\mathfrak{F} \in D^{-}_{\coh}(\hat{\cX})$.  In
  general, let $\mathfrak{F} \in D_{\coh}(\hat{\cX})$. Set
  $\cG_n = c_{\qcoh,*}(\tau^{\leq n}\mathfrak{F}) \in D^{\leq
    n}_{\Coh}(\cX)$. Let $\cG = \mathrm{hocolim}_n \cG_n$; then
  $\mathscr{H}^q(\cG) \simeq \mathscr{H}^q(\cG_n) \in \Coh(\cX)$
  whenever $q\leq n$. In particular, $\cG \in D_{\Coh}(\cX)$. But \cite[Tag
  \spref{0949}]{ref:stacks} implies that
  \[
    c^*\cG = c^*\mathrm{hocolim}_n \cG_n \simeq \mathrm{hocolim}_n
    c^*\cG_n \simeq \mathrm{hocolim}_n \tau^{\leq n} \mathfrak{F}
    \simeq \mathfrak{F}.
  \]
  Arguing as before, the result follows.
\end{proof}
The following two lemmas featured in the proof of Proposition
\ref{P:underived-derived-comparison}.
\begin{lemma}\label{L:trunc-holim-formal}
  Let $\cX$ be a noetherian algebraic stack and $\cX_0 \subset \cX$ a
  closed substack. If $\mathfrak{F} \in D_{\coh}(\hat{\cX})$, then
  $\mathfrak{F} \simeq \holim{n} \tau^{\geq -n}\mathfrak{F}$.
\end{lemma}
\begin{proof}
  This follows from \cite[Tag \spref{0D6S}]{ref:stacks} and Serre
  vanishing for coherent sheaves on affine noetherian formal schemes.
\end{proof}
\begin{lemma}\label{L:coherent-gens-der}
  Let $\cX$ be a noetherian algebraic stack. Let $f \colon \cF \to \mathcal{G}$ be
  a morphism in $D^+_{\qcoh}(\cX)$. If $\RDERF \Hom_{\oh_{\cX}}(\mathcal{P},f)$
  is an isomorphism for all $\mathcal{P} \in \Coh(\cX)$, then $f$ is an
  isomorphism.
\end{lemma}
\begin{proof}
  Let $\mathcal{C}$ be a cone for $f$. If $\mathcal{C} \neq 0$, choose $d\in \bZ$ least
  such that $\mathscr{H}^d(\mathcal{C}) \neq 0$. Then there is a non-zero
  coherent sheaf $\mathcal{P}$ and an injection $\mathcal{P} \subseteq \mathscr{H}^d(\mathcal{C})$
  \cite[Prop.\ 15.4]{ref:lmb}. Hence, the induced map $\mathcal{P}[-d] \to \mathcal{C}$ is
  non-zero, which is a contradiction.
\end{proof}

As a result, we can conclude that formal functions implies that $\Coh(\cX) \to \Coh(\widehat{\cX})$ is fully faithful.

\begin{corollary} \label{C:fullyfaithful}
  Let $\cX$ be a noetherian algebraic stack and $\cX_0 \subset \cX$ a closed substack. If $(\cX, \cX_0)$ satisfies formal functions, then the functor
  \[ 
    \Coh(\cX) \to \Coh(\widehat{\cX}), \qquad \cF \mapsto \hat{\cF}
  \]
  is fully faithful, with essential image stable under kernels,
  cokernels and extensions. Moreover, it is an equivalence if for each non-zero
  $\mathfrak{F} \in \Coh(\widehat{\cX})$ there exist $\cF \in \Coh(\cX)$
  and a non-zero map $c^*\cF \to \mathfrak{F}$.
\end{corollary}
\begin{proof}
  Since $(\cX,\cX_0)$ satisfies formal functions, it satisfies derived
  formal functions (\Cref{P:underived-derived-comparison}\eqref{PI:underived-derived-comparison:ff}). Hence,
  $\Coh(\cX) \to\Coh(\hat{\cX})$ is exact and fully faithful. It follows immediately that the image of $\Coh(\cX)$ in $\Coh(\hat{\cX})$ is stable under kernels and cokernels. For the extensions, we know that derived formal functions implies that if $\cF$, $\cG \in \Coh(\cX)$, then
  \[
    \Ext^1_{\oh_{\cX}}(\cF,\cG) = \Hom_{\oh_{\cX}}(\cF,\cG[1]) \simeq \Hom_{\oh_{\hat{\cX}}}(\hat{\cF},\hat{\cG}[1]) =     \Ext^1_{\oh_{\hat{\cX}}}(\hat{\cF},\hat{\cG}). 
  \]
  Hence, the image is also stable under extensions.

  If $\mathfrak{P} \in \modd(\widehat{\cX})$, let
  $c_{\qcoh,*}^0(\mathfrak{P}) =
  \mathscr{H}^{0}(c_{\qcoh,*}\mathfrak{P})$, which is a left-exact
  functor. For the equivalence, the condition implies that
  $c^0_{\qcoh,*}(\mathfrak{P}) \neq 0$ whenever
  $\mathfrak{P} \in \Coh(\widehat{\cX})$. Now fix
  $\mathfrak{F} \in \Coh(\widehat{\mathfrak{F}})$. If
  $\cG \subseteq c^0_{\qcoh,*}(\mathfrak{F})$ is a coherent subsheaf,
  let $\mathfrak{K} = \ker(c^*\cG \to \mathfrak{F})$. Then
  $\mathfrak{K} \in \Coh(\widehat{\cX})$ and there is an exact
  sequence:
  \[
    0 \to c^0_{\qcoh,*}(\mathfrak{K}) \to c^0_{\qcoh,*}c^*\cG \to c^0_{\qcoh,*}\mathfrak{F}.
  \]
  By
  \Cref{P:underived-derived-comparison}\eqref{PI:underived-derived-comparison:ff},
  $\cG \to c_{\qcoh,*}^0c^*\cG$ is an isomorphism. Hence,
  $c^0_{\qcoh,*}(\mathfrak{K}) = 0$ and so $\mathfrak{K} = 0$. Now
  write $c_{\qcoh,*}^0\mathfrak{F} = \cup_\lambda \cF_\lambda$ as a
  filtered union of coherent subsheaves. Then
  $\{ c^*\cF_\lambda \}_\lambda$ defines an increasing union of coherent
  subsheaves of $\mathfrak{F}$. By coherence of $\mathfrak{F}$, there
  must be a $\lambda_0$ such that $c^*\cF_\lambda = c^*\cF_{\lambda_0}$
  for all $\lambda \geq \lambda_0$. By full faithfulness,
  $\cF_\lambda = \cF_{\lambda_0}$ for all $\lambda \geq \lambda_0$ and so
  $c_{\qcoh,*}^0\mathfrak{F}$ is coherent. Finally, consider the exact
  sequence:
  \[
    0 \to c^*c_{\qcoh,*}^0\mathfrak{F} \to \mathfrak{F} \to \mathfrak{C} \to 0. 
  \]
  Applying $c_{\qcoh,*}^0$, we see that from \Cref{P:underived-derived-comparison}\eqref{PI:underived-derived-comparison:ff} again that
  \[
    0 \to c_{\qcoh,*}^0\mathfrak{F} \xrightarrow{\mathrm{id}}
    c_{\qcoh,*}^0\mathfrak{F} \to c_{\qcoh,*}^0\mathfrak{C} \to 0
  \]
  is exact. It follows that $c_{\qcoh,*}^0\mathfrak{C} = 0$ and so
  $\mathfrak{C} = 0$, which proves the desired equivalence.
\end{proof}
\section{Strong cohomological properness and formal functions}\label{S:cp-implies-ff}
The main result of this section establishes a relationship between
formal functions and a strengthening of cohomological
properness. While something similar to \Cref{cor:adic-cp-ff} appeared
in \cite{ref:hlp}, \Cref{cor:reductive-ff} is new in positive and
mixed characteristic (it was established in the linearly reductive
case in \cite[Thm.~1.6]{ref:ahr2}). The arguments we use are simple
generalizations of the related results of \cite[III]{ref:ega} and \cite[IX]{ref:sga2}. To this
end, we make the following definition.
\begin{definition}\label{D:strong-cp}
  Let $f\colon \cX \to \cY$ be a morphism of noetherian algebraic
  stacks, $\cY_0 \subset \cY$ a closed substack,
  and $\cX_0 \subset f^{-1}(\cY_0)$ a closed substack. We say that the
  triple $(f,\cX_0,\cY_0)$ is \emph{strongly cohomologically proper}
  if for sheaves of ideals $\cJ$ and $\cI$ defining
  $\cY_0 \subset \cY$ and $\cX_0 \subset \cX$, the morphism
    $\underline{\Spec}_{\cO_{\cX}}(\oplus_{n\geq 0} \cI^n) \to
    \underline{\Spec}_{\cO_{\cY}}(\oplus_{n\geq
      0}\cJ^n)$ is cohomologically proper.
\end{definition}
\begin{remark}\label{R:strong-cp}
  In \Cref{D:strong-cp}, from the commutative diagram:
  \[
    \xymatrix{\cX \ar@{^(->}[r] \ar[d]_f &
      \underline{\Spec}_{\cO_{\cX}}(\oplus_{n\geq 0} \cI^n) \ar[d] \\
      \cY \ar@{^(->}[r] & \underline{\Spec}_{\cO_{\cY}}(\oplus_{n\geq
        0}\cJ^n),}
  \]
  where the horizontal morphisms are closed immersions, it follows
  that $f$ is cohomologically proper. It is also easy to see that the
  definition only depends on the closed subsets
  $|\cX_0| \subseteq |\cX|$ and $|\cY_0| \subseteq |\cY|$ (not the
  ideals $\cI$ and $\cJ$).  Indeed, suppose that we have other ideals
  $\cI'$, $\cJ'$ with vanishing locus $|\cX_0|$ and $|\cY_0|$. Now
  form the commutative diagram:
  \[
    \xymatrix{\underline{\Spec}_{\cO_{\cX}}(\oplus_{n\geq 0} \cI^n) \ar[d] \ar[r] &\underline{\Spec}_{\cO_{\cX}}(\oplus_{n\geq 0} (\cI\cap \cI')^n) \ar[d]  & \underline{\Spec}_{\cO_{\cX}}(\oplus_{n\geq 0} \cI'^n) \ar[d] \ar[l] \\
      \underline{\Spec}_{\cO_{\cY}}(\oplus_{n\geq 0}\cJ^n) \ar[r] &
      \underline{\Spec}_{\cO_{\cY}}(\oplus_{n\geq 0}(\cJ \cap \cJ')^n)
      & \underline{\Spec}_{\cO_{\cY}}(\oplus_{n\geq 0}\cJ'^n) \ar[l].}
  \]
  Since $\cX$ and $\cY$ are noetherian, it is easy to see that the
  horizontal morphisms are all finite and surjective. In particular, the
  cohomological properness of any of the vertical morphisms implies
  that of the others.
\end{remark}
\begin{remark}\label{R:strongcp-stein}
  Let $f\colon \cX \to \cY$ be a morphism of noetherian algebraic
  stacks with affine diagonal, $\cY_0 \subset \cY$ a closed substack,
  and $\cX_0 \subset f^{-1}(\cY_0)$ a closed substack. Assume that
  $\cO_{\cY} \to f_*\cO_{\cX}$ is an isomorphism (i.e., $f$ is Stein)
  and that $\cJ \simeq f_*\cI$, where $\cJ$ defines
  $\cY_0 \subset \cY$ and $\cI$ defines $\cX_0 \subset \cX$. Then the
  triple $(f,\cX_0,\cY_0)$ is strongly cohomologically proper
  if and only if the $\cO_{\cY}$-algebra $\oplus_{n\geq 0} f_*(\cI^n)$
  is finitely generated and the induced morphism
  $\underline{\Spec}_{\cO_{\cX}}(\oplus_{n\geq 0} \cI^n) \to
  \underline{\Spec}_{\cO_{\cY}}(\oplus_{n\geq 0} f_*(\cI^n))$ is
  cohomologically proper. Indeed, we have the composition
  \[
    \underline{\Spec}_{\cO_{\cX}}(\oplus_{n\geq 0} \cI^n) \to
    \underline{\Spec}_{\cO_{\cY}}(\oplus_{n\geq 0} f_*(\cI^n)) \to
    \underline{\Spec}_{\cO_{\cY}}(\oplus_{n\geq 0} \cJ^n) \to \cY.
  \]
  If $(f,\cX_0,\cY_0)$ is strongly cohomologically proper, then
  $\oplus_{n\geq 0} f_*(\cI^n)$ is a finite
  $\oplus_{n\geq 0} \cJ^n$-algebra and so the induced morphism
  $\underline{\Spec}_{\cO_{\cX}}(\oplus_{n\geq 0} \cI^n) \to
  \underline{\Spec}_{\cO_{\cY}}(\oplus_{n\geq 0} f_*(\cI^n))$ is
  cohomologically proper. Conversely, if $\oplus_{n\geq 0}f_*(\cI^n)$
  is a finitely generated $\cO_{\cY}$-algebra, then there is an
  $N\gg 0$ such that $f_*(\cI^N)^k = f_*(\cI^{Nk})$ as ideals of
  $f_*\cO_{\cX} \simeq \cO_{\cY}$ for all $k\geq 0$. But we have
  $\cJ^N \subseteq f_*(\cI^N) \subseteq f_*(\cI) = \cJ$. In particular, we have a commutative diagram:
  \[
    \xymatrix{\underline{\Spec}_{\cO_{\cX}}(\oplus_{n\geq 0} \cI^{n}) \ar[r] \ar[d] & \underline{\Spec}_{\cO_{\cX}}(\oplus_{k\geq 0} \cI^{Nk}) \ar[d] \\
    \underline{\Spec}_{\cO_{\cY}}(\oplus_{n\geq 0}f_*(\cI^{n})) \ar[d] \ar[r] & \underline{\Spec}_{\cO_{\cY}}(\oplus_{k\geq 0}f_*(\cI^{N})^k)\ar[d] \\ \underline{\Spec}_{\cO_{\cY}}(\oplus_{n\geq 0} \cJ^{n}) \ar[r] & \underline{\Spec}_{\cO_{\cY}}(\oplus_{k\geq 0} \cJ^{Nk}),}
\]
where the horizontal arrows are all finite and surjective. The top vertical arrow on
the right is cohomologically proper and the bottom vertical arrow on
the right is finite, whence the composition is cohomologically
proper. It follows that the composition of the vertical arrows on the
left is cohomologically proper, which proves the claim.
\end{remark}
We have the following two key examples.
\begin{example}\label{E:adic-strong-cp}
  Let $f\colon \cX \to \cY$ be a universally cohomologically proper
  morphism of noetherian algebraic stacks. Let
  $\cY_0 \subset \cY$ be a closed substack. Then the triple
  $(f,f^{-1}(\cY_0), \cY_0)$ is strongly cohomologically
  proper. To see this: let $\cJ \subset \cO_{\cY}$ be a coherent
  sheaf of ideals definining $\cY_0 \subset \cY$. Since $\cY$ is
  noetherian, $\oplus_{n\geq 0} \cJ^n$ is a finitely generated
  $\cO_{\cY}$-algebra. Now form the cartesian square:
  \[
    \xymatrix{\underline{\Spec}_{\cO_{\cX}}(\oplus_{n\geq 0} \cJ^n\cO_{\cX})
      \ar[r] \ar[d]_{\tilde{f}} & \cX \ar[d]^f
      \\\underline{\Spec}_{\cO_{\cY}}(\oplus_{n\geq 0} \cJ^n) \ar[r] & \cY.}
  \]
  Since $f$ is universally cohomologically proper, $\tilde{f}$ is
  cohomologically proper. 
\end{example}
\begin{example}\label{E:fund-strong-cp}
  Let $R$ be a noetherian ring. Let $A$ be a finitely generated
  $R$-algebra with an action of a reductive group scheme $G$ over
  $\Spec R$. Let $I \subseteq A$ be a $G$-equivariant ideal. Let
  $\cX=[\Spec A/G]$, $\cX_0 = [\Spec (A/I)/G]$,
  $\cY_0 = \Spec (A^G/I^G)$ and $f\colon \cX \to \Spec (A^G)$. Then
  the triple $(f, \cX_0,\cY_0)$ is strongly cohomologically
  proper. Indeed, by \Cref{E:fund-is-cp}, $f$ is universally
  cohomologically proper. Further, the $G$-equivariant $A$-algebra
  $\oplus_{n\geq 0} I^n$ is finitely generated and so
  $(\oplus_{n\geq 0}I^n)^G = \oplus_{n\geq 0} \mathrm{H}^0([\Spec
  A/G],\tilde{I}^n)$ is a finitely generated $A^G$-algebra, where
  $\tilde{I}$ denotes the coherent sheaf of ideals on $[\Spec A/G]$
  associated to $I$. By \Cref{E:fund-is-cp} again, the induced
  morphism
  $[\Spec (\oplus_{n\geq 0} I^n)/G] \to \Spec ((\oplus_{n\geq
    0}I^n)^G)$ is cohomologically proper and the claim follows from
  \Cref{R:strongcp-stein}. If $A^G$ is artinian, this also follows from
  \Cref{thm:main}.
\end{example}
We now have the main result of this section. 
\begin{theorem} \label{T:strong-cp-implies-ff} Let
  $f \colon \cX \to \Spec R$ be a finite type morphism of noetherian
  algebraic stacks, $J \subset R$ an ideal, and
  $\cX_0 \subset f^{-1}(\Spec (R/J))$ a closed substack. Assume that
  the triple $(f,\cX_0,\Spec(R/J))$ is strongly cohomologically proper and
  $R$ is $J$-adically complete.
  \begin{enumerate}[{\rm(1)}]
  \item \label{TI:strong-cp-implies-ff:complete-fns} Let $\cI$ be a
    coherent $\cO_{\cX}$-ideal defining the closed immersion
    $\cX_0 \subset \cX$. Then $\Gamma(\cX,\cO_{\cX})$ is
    $\Gamma(\cX,\cI)$-adically complete.
  \item \label{TI:strong-cp-implies-ff:ff} Formal functions holds for
    the pair $(\cX,\cX_0)$. Moreover, if $\cF$ is a coherent
    $\cO_{\cX}$-module, then the natural morphisms:
    \[
      \mathrm{H}^q(\cX,\cF) \to \mathrm{H}^q(\hat{\cX},\hat{\cF}) \to
      \varprojlim_n \mathrm{H}^q(\cX_n,\cF_n)
    \]
    are isomorphisms for all $q\geq 0$. 
  \item \label{TI:strong-cp-implies-ff:fcp}If $\mathfrak{F}$ is a
    coherent $\cO_{\hat{\cX}}$-module, then
    $\mathrm{H}^q(\hat{\cX},\mathfrak{F})$ is a finitely generated
    $R$-module and the natural morphism:
    \[
      \mathrm{H}^q(\hat{\cX},\mathfrak{F}) \to \varprojlim_n
      \mathrm{H}^q(\cX_n,\mathfrak{F}_n)
    \]
    is an isomorphism for all $q\geq 0$.
  \end{enumerate}
\end{theorem}
Before we prove \Cref{T:strong-cp-implies-ff}, we provide several
corollaries.
\begin{corollary}\label{cor:ucp-stein}
  Let $f\colon \cX \to \cY$ be a universally cohomologically proper
  morphism of noetherian algebraic stacks. Then
  $f$ admits a \emph{Stein factorization}:
  $$\cX \xrightarrow{\bar{f}}
  \underline{\Spec}_{\cO_{\cY}}(f_*\cO_{\cX}) \xrightarrow{a} \cY,$$
  where $a$ is finite and $\bar{f}$ is universally cohomologically
  proper and universally closed with geometrically connected fibers.
\end{corollary}
\begin{proof}
  By \Cref{R:ucp-uc}, it suffices to prove that if $\cY= \Spec R$,
  where $(R,\mathfrak{m})$ is an $\mathfrak{m}$-adically complete
  noetherian local ring and $R \simeq \mathrm{H}^0(\cX,\cO_{\cX})$;
  then the closed fiber
  $\cX_0=\cX\otimes_R R/\mathfrak{m} \to \Spec (R/\mathfrak{m})$ is
  connected. By \Cref{E:adic-strong-cp}, the triple $(\cX,\cX_0, f)$
  is strongly cohomologically proper. It follows from
  \Cref{T:strong-cp-implies-ff}\eqref{TI:strong-cp-implies-ff:ff} that
  $R\simeq \mathrm{H}^0(\cX,\cO_{\cX}) \simeq \varprojlim_n
  \mathrm{H}^0(\cX_n,\cO_{\cX_n})$. If $\cX_0$ is disconnected, then
  the corresponding idempotent $e_0$ lifts to a non-trivial idempotent
  $e$ of $R$, which contradicts $R$ being local \cite[Tag
  \spref{0G7X}]{ref:stacks}.
\end{proof}
\begin{corollary}\label{cor:adic-cp-ff}
  Let $\pi \colon \cX \to \Spec R$ be a {universally}
  cohomologically proper morphism of noetherian algebraic stacks. Let
  $J \subseteq R$ be an ideal. If $R$ is $J$-adically complete, then
  formal functions holds for the pair $(\cX,\pi^{-1}(\spec R/J))$.
\end{corollary}
\begin{proof}
  Combine \Cref{E:adic-strong-cp} with \Cref{T:strong-cp-implies-ff}.
\end{proof}
\begin{corollary}\label{cor:reductive-ff}
  Let $R$ be a noetherian ring. Let $A$ be a finitely generated
  $R$-algebra with an action of a reductive group scheme $G$ over
  $\Spec R$. Let $I \subseteq A$ be a $G$-equivariant ideal. Let
  $\cX = [\Spec A/G]$, $\cX_0 = [\Spec (A/I)/G]$,
  $\cY_0 = \Spec (A^G/I^G)$, $\cY=\Spec (A^G)$ and
  $f \colon \cX \to \cY$ the induced morphism. If $A^G$ is
  $I^G$-adically complete, then the triple $(f,\cX_0,\cY_0)$ is
  strongly cohomologically proper and the pair $(\cX,\cX_0)$ satisfies
  formal functions.
\end{corollary}
\begin{proof}
  Combine \Cref{E:fund-strong-cp} and \Cref{T:strong-cp-implies-ff}.
\end{proof}
We recall some background on filtrations that will be important for
the proof of \Cref{T:strong-cp-implies-ff}. Let $A$ be a ring,
$I \subset A$ an ideal, and $M$ an $A$-module. A filtration
$(M_n)_{n\in \mathbf{Z}}$ of $M$ is $I$-\emph{good} (or
$I$-\emph{stable}) if the following three conditions are satisfied:
\begin{enumerate}
\item $M=M_k$ for some $k\in \mathbf{Z}$;
\item $IM_n \subset M_{n+1}$ for all $n\in \mathbf{Z}$; and 
\item $IM_n = M_{n+1}$ for all $n\gg 0$.
\end{enumerate}
Obviously, the filtration $(I^{n+1}M)_{n\geq -1}$ is $I$-good; the
topology that this filtration defines on $M$ is called the $I$-adic
topology on $M$. A key observation is that the topology on $M$ defined
by any $I$-good filtration is equivalent to the $I$-adic topology on
$M$ \cite[Lem.~10.6]{MR0242802}. A much deeper fact is that if $A$ is
noetherian and $M$ is finitely generated, then a filtration $(M_n)$ is
$I$-stable if and only if $M_\ast = \oplus_{n\in \mathbf{Z}} M_n$ is a
finitely generated $A_\ast = \oplus_{n\geq 0} I^n$-module
\cite[Lem.~10.8]{MR0242802}. A key consequence of this whole theory is
that if $M$ is a finitely generated $A$-module and $A$ is $I$-adically
complete, then $M$ is $I$-adically complete.

Assume now that we are in the situation of
\Cref{T:strong-cp-implies-ff}.  Let $\cM$ be a coherent
$\cO_\cX$-module. Let $\cI_\ast = \oplus_{n\geq 0} \cI^n$ and let
$\cM_\ast = \oplus_{n\geq 0} \cI^{n}\cM$. The quasi-coherent
$\cO_\cX$-algebra $\cI_\ast$ is of finite type and $\cM_\ast$ is a
coherent $\cI^\ast$-module \cite[Lem.~10.8]{MR0242802}.
    
Let $I_\ast=\Gamma(\cX,\cI_\ast)$ and
$M^q_{\ast} = \mathrm{H}^q(\cX,\cM_\ast)$. Let us briefly remark on
the graded structure of $M^q_\ast$. If
$x \in I_s = \Gamma(\cX,\cI^s)$, then for all $t\geq 0$ there is an
induced homomorphism of $\cO_\cX$-modules that
$\cI^t\cM \to \cI^{s+t}\cM$ that is multiplication by $x$. It follows
that we obtain an induced morphism:
  \[
    \mu_{x,\cM,t}^q \colon
    \mathrm{H}^q(\cX,\cI^t\cM) \to
    \mathrm{H}^q(\cX,\cI^{t+s}\cM).
  \]
  This is how $M^q_{\ast}$ becomes a graded $I_\ast$-module. In
  particular, $I_sM^q_t \subset M^q_{t+s}$ denotes the image of the
  natural $R=I_0$-module homomorphism
  $I_s\tensor_R M^q_t \to M^q_{t+s}$.

  Further, the canonical inclusions
  $\cI^{t}\cM \subset \cI^{t'}\cM$
  for $t\geq t'$ give rise to an inverse system
  $(\mathrm{H}^q(\cX,\cI^t\cM))_{t\geq 0}$ with
  transition map
  $\nu_{\cM,t,t'}^q \colon
  \mathrm{H}^q(\cX,\cI^{t}\cM) \to
  \mathrm{H}^q(\cX,\cI^{t'}\cM)$ when $t\geq t'$. It
  follows that the composition:
  \[
    \mathrm{H}^q(\cX,\cI^t\cM)
    \xrightarrow{\mu_{x,\cM,t}^q}
    \mathrm{H}^q(\cX,\cI^{t+s}\cM)
    \xrightarrow{\nu_{\cM,t+s,t}^q}
    \mathrm{H}^q(\cX,\cI^t\cM)
  \]
  coincides with multiplication by $x$ on
  $\mathrm{H}^q(\cX,\cI^t\cM)$ as an $R$-module. In
  particular, if
  $P \subset \mathrm{H}^q(\cX,\cI^{t}\cM)$ is an
  $R$-submodule, then we have the equality of $R$-submodules of
  $\mathrm{H}^q(\cX,\cI^t\cM)$:
  \begin{equation}
    \Gamma(\cX,\cI^{s})P = \nu^q_{\cM,t+s,t}(I_sP) \label{eq:4}.
  \end{equation}
  By assumption, $I_\ast$ is a finitely generated $R$-algebra and
  $M^q_\ast$ is a finitely generated and graded $I_\ast$-module. It
  follows that for some sufficiently large $N$, that as ideals of the
  finite $R$-algebra $\Gamma(\cX,\cO_{\cX})$, we have
  $\Gamma(\cX,\cI^N)^k = \Gamma(\cX,\cI^{Nk})$ for all $k\geq 0$. In
  particular, replacing $I$ and $\cI$ by $I^N$ and $\cI^N$,
  respectively, we may assume that
  $\Gamma(\cX,\cI)^k=\Gamma(\cX,\cI^k)$ for all $k\geq 0$.
\begin{proof}[Proof of \Cref{T:strong-cp-implies-ff}]
  We first prove \eqref{TI:strong-cp-implies-ff:complete-fns}. The
  discussion above showed that $\oplus_{n\geq 0} I^n$ is a finite
  $\oplus_{n\geq 0} J^n$-module. It follows that the $I$-adic
  filtration on $S$ is $J$-stable and so the $J$-adic topology is
  equivalent to the $I$-adic topology. Since $R$ is $J$-adically
  complete and $S$ is a finite $R$-module, $S$ is $I$-adically
  complete too. 
  
  The proof of \eqref{TI:strong-cp-implies-ff:ff} follows Serre's
  argument in \cite[III.4.1.5]{ref:ega} (cf.\ \cite[\S 8.2]{MR2222646}
  and \cite[\S4]{ref:ahr2}). By
  \eqref{TI:strong-cp-implies-ff:complete-fns}, we may assume that
  $R=\Gamma(\cX,\cO_{\cX})$ and
  $\Gamma(\cX,\cI^k) = \Gamma(\cX,\cI)^k$ for all $k\geq 0$ and so
  $I_\ast = \oplus_{k\geq 0} I^k=R_\ast$. By \Cref{R:formal-fns}, it
  suffices to prove that
  $\mathrm{H}^q(\cX,\cF) \to \varprojlim_n \mathrm{H}^q(\cX_n,\cF_n)$
  is an isomorphism for all coherent sheaves $\cF$ on $\cX$ and
  $q\geq 0$ and $\{H^{q-1}(\cX_n,\cF_n)\}_n$ satisfies the Artin--Rees
  condition (this implies that the $\lim^1$ term in
  \eqref{eq:milnor-sequence} vanishes).

  We now let $\cF$ be a coherent $\cO_\cX$-module. Let
  $q\geq 0$ and $n\geq -1$ and consider the exact sequence of
  $R$-modules:
  \begin{equation}
    \xymatrix{0 \ar[r] & R^q_n \ar[r] & H^q \ar[r] & H^q_n \ar[r] & Q^q_n \ar[r] & 0,}\label{eq:2}
  \end{equation}  
  where $H^q = \mathrm{H}^q(\cX,\cF)$,
  $H^q_n = \mathrm{H}^q(\cX,\cF/\cI^{n+1}F)$,
  $L^q_{n} = \mathrm{H}^q(\cX,\cI^{n+1}\cF)$,
  \begin{align*}
    R^q_n &= \ker(H^q \to H^q_n)  = \mathrm{im}(L^q_n \to H^q), \,\mbox{and}\\
    Q^q_n &= \mathrm{im}(H^q_n \to {L}^{q+1}_n) = \ker({L}^{q+1}_n \to H^{q+1}).
  \end{align*}
  The result follows from the following three claims:
  \begin{enumerate}
  \item the filtration $(R^q_n)_{n\geq -1}$ on $H^q$ is $I$-good;
  \item the inverse system $(Q^q_n)$ is Artin--Rees zero (i.e., there exists an $s$ such that $Q_{n+s}^q \to Q^q_n$ is $0$ for all $n$); 
  \item the inverse system $(H^{q-1}_n)$ satisfies the uniform
    Artin--Rees condition (i.e., there is an $s$ such that the images
    of the morphisms $H^{q-1}_{n'} \to H^{q-1}_{n}$ agree for all
    $n'\geq n+s$).
  \end{enumerate}
  Indeed, the exact sequence \eqref{eq:2} induces the following short
  exact sequence:
  \begin{equation}
    \xymatrix{0 \ar[r] & H^q/R^q_n \ar[r] & H^q_n \ar[r] & Q^q_n \ar[r] &
      0.}\label{eq:3}
  \end{equation}
  We now take inverse limits, and obtain the following exact sequence:
  \[
    \xymatrix{0 \ar[r] & \varprojlim_n H^q/R^q_n \ar[r] & \varprojlim_n
      H^q_n \ar[r] & \varprojlim_n Q^q_n. }
  \]
  Since the system $(Q^q_n)$ is Artin--Rees zero, it follows
  immediately that $\varprojlim_n Q^q_n = 0$. Moreover, the filtration
  $(R^q_n)$ on $H^q$ is $I$-good and since $H^q$ is a finitely
  generated $R$-module, it follows that the natural map
  $H^q \to \varprojlim_n H^q/R^q_n$ is an isomorphism. What results
  from all of this is an isomorphism:
  \[
    \mathrm{H}^q(\cX,\cF) \simeq \varprojlim_n
    \mathrm{H}^q(\cX,\cF/\cI^{n+1}\cF).
  \]
  That is, Zariski formal functions holds for the pair $(\cX,\cX_0)$.
  
  We first establish that the filtration $(R^q_n)_{n\geq -1}$ on $H^q$ is
  $I$-good. To see this, we first
  note that $R_{-1} = H^q$. We now apply the previous discussion to the
  $\cO_X$-module $\cI\cF$. It follows that
  ${L}^q_\ast=\oplus_{n\geq 0} {L}^q_n$ is
  a finitely generated $R_\ast$-module. But the graded $R_\ast$-module
  $\oplus_{n\geq 0}R^q_n$ is the image of the graded $R_\ast$-module
  homomorphism
  \[
    \oplus_{n\geq 0} {L}^q_n \to
    \oplus_{n\geq 0} H^q,
  \]
  and so $\oplus_{n\geq -1} R^q_n$ is also a finitely generated graded
  $I_\ast$-module. By \cite[Lem.~10.8]{MR0242802}, it follows that the
  filtration $(R^q_n)_{n\geq -1}$ is $I$-good.

  We next prove that the inverse system $(Q^q_n)$ is Artin--Rees
  zero. First observe that $Q_\ast^q = \oplus_{n\geq 0} Q_n^q$ is a
  $R_\ast$-submodule of $L_\ast^{q+1} = \oplus_{n\geq 0} L_n^{q+1}$,
  which is a finitely generated $R_\ast$-module. Hence, $Q_\ast^q$ is
  a finitely generated $R_\ast$-module. In particular, there exist
  integers $h$, $l\geq 0$ such that $I_hQ_k^q = Q_{h+k}^q$ for all
  $k\geq l$. But $Q_n^q$ is always a quotient of $H^q_n$ and $H^q_n$
  is annihilated by $I^{n+1}$ and so if $m\geq l+h$, then
  write $m=th+r+l$, where $0\leq r<h$. Then
  \[
    I^{l+h+1}Q_{m}^q = I^{l+h+1}I_{th}Q_{l+r} \subset I^{l+r+1}I_{th}Q_{l+r} =0.
  \]
  It follows from \eqref{eq:4} and the above that if
  $s=(h+2)(l+h)\geq l+h+1$, then for $t\geq 0$ we have
  \[
    \nu_{\cI\cF,t+s,t}^{q+1}(Q^q_{t+s}) =
    \nu_{\cI\cF,t+s,t}^{q+1}(I_{(h+1)(l+h)}Q^q_{t+h+l})
    = I^{(h+1)(l+h)}Q^q_{t+l+h} \subset I^{l+h+1}Q^q_{t+l+h} = 0.
  \]
  Finally, the exact sequence of \eqref{eq:3} and basic properties of
  the Artin--Rees condition shows that it suffices to prove that the
  inverse systems $(H^q/R_n)$ and $(Q_n)$ satisfy the uniform
  Artin--Rees condition. Since $(Q_n)$ is Artin--Rees zero, it
  satisfies the uniform Artin--Rees condition. Further, since
  $(H^q/R_n)$ is a surjective system it trivially satisfies the uniform
  Artin--Rees condition. This proves \eqref{TI:strong-cp-implies-ff:ff}

  We now prove \eqref{TI:strong-cp-implies-ff:fcp}. Note that
  $\oplus_{n\geq 0} \cI^n/\cI^{n+1}$ is a coherent
  $\oplus_{n\geq 0} \cI^n$-algebra. In particular, the cohomological
  properness of
  $\underline{\Spec}_{\cO_{\cX}}(\oplus_{n\geq 0} \cI^n) \to
  \underline{\Spec}_{\cO_{\cY}}(\oplus_{n\geq 0} f_*(\cI^{n}))$
  implies that $\oplus_{n\geq 0} f_*(\cI^n/\cI^{n+1})$ is a coherent
  $\oplus_{n\geq 0} f_*(\cI^{n})$-algebra. It follows that:
    \begin{enumerate}[{\rm(1)}]
    \item $\oplus_{n\geq 0} \mathrm{H}^0(\cX_0,\cI^n/\cI^{n+1})$ is a
      finitely generated $\Gamma(\cX,\cO_{\cX})$-algebra; and
    \item
    $\underline{\Spec}_{\cO_{\cX_0}}(\oplus_{n\geq 0} \cI^n/\cI^{n+1})
    \to \Spec (\oplus_{n\geq 0}\mathrm{H}^0(\cX_0,\cI^n/\cI^{n+1}))$
    is cohomologically proper.
  \end{enumerate}
  The arguments of \cite[III.3.4.4]{ref:ega} now apply verbatim to
  prove \eqref{TI:strong-cp-implies-ff:fcp}.
\end{proof}
\section{Permanence of properties}\label{sec:permanence}
Let $f\colon \cX \to \cY$ be a morphism of noetherian algebraic
stacks, $\cY_0 \subset \cY$ a closed substack, and
$\cX_0 \subset f^{-1}(\cY_0)$ a closed substack. Then there is an
induced diagram:
\[
\xymatrix@+1pc{\modd(\cX) \ar@<-1ex>[r]_{c^*_{\cX}}\ar@<-1ex>[d]_{f_*} & \modd(\hat{\cX}) \ar@<-1ex>[l]_{c_{\cX,*}} \ar@<-1ex>[d]_{\hat{f}_*} \\ \modd(\cY) \ar@<-1ex>[r]_{c^*_{\cY}} \ar@<-1ex>[u]_{f^*}& \modd(\hat{\cY}) \ar@<-1ex>[u]_{\hat{f}^*} \ar@<-1ex>[l]_{c_{\cY,*}},}
\]
as well as natural isomorphisms of functors
$c_{\cX}^*f^* \simeq \hat{f}^*c_{\cY}^*$ and
$f_*c_{\cX,*} \simeq c_{\cY,*}\hat{f}_*$.  Because of the lack of
functoriality of the lisse-\'etale topos, the left derived functors of
$f^*$ and $\hat{f}^*$ are somewhat subtle if $f$ is not smooth. The
derived functors on the level of unbounded derived categories
$\RDERF f_* \colon D(\cX) \to D(\cY)$ and
$\RDERF \hat{f}_*\colon D(\hat{\cX}) \to D(\hat{\cY})$ always exist,
however. As $c_{\cX,*}$ and $c_{\cY,*}$ are exact (they are just
forgetful functors), it follows that we also have natural isomorphisms
$(\RDERF f_*)\circ c_{\cX,*} \simeq c_{\cY,*}\circ\RDERF\hat{f}_*$. If
$\cF \in D(\cX)$, then there is thus always a comparison morphism:
\begin{equation} \label{eqn:comparison-morphism} (\RDERF
  f_*\cF)^{\wedge} = c_{\cY}^*(\RDERF f_*\cF) \to \RDERF
  \hat{f}_*c_{\cX}^*\cF = \RDERF \hat{f}_* \hat{\cF},
\end{equation}
which comes from the adjoint of the morphism:
\[
  \RDERF f_*\cF \to (\RDERF f_*)\circ c_{\cX,*}\circ c_{\cX}^*\cF
  \simeq c_{\cY,*} \circ \RDERF \hat{f}_* \circ c_{\cX}^*\cF.
\]
We now introduce a relative version of \Cref{def:ff-cc} (cf.\
\cite[Defn.\ 1.1.3]{ref:hlp}).
\begin{definition}
  Let $f\colon \cX \to \cY$ be a morphism of noetherian algebraic
  stacks, $\cY_0 \subset \cY$ is a closed substack, and
  $\cX_0 \subset f^{-1}(\cY_0)$ is closed substack. We say that the
  triple $(f,\cX_0,\cY_0)$ satisfies \emph{relative formal
    functions} (resp.\ \emph{relatively (derived) coherently
    complete}) if $(\cX_{\hat{R}},(\cX_0)_{\hat{R}})$ satisfies formal
  functions (resp.\ (derived) coherent completeness) for all
  noetherian rings $R$ and \emph{smooth} morphisms
  $\rho\colon \Spec R \to \cY$, where $\hat{R}$ denotes the adic
  completion of $R$ with respect to the ideal defining
  $\rho^{-1}(\cY_0) \subset \Spec R$,
  $\cX_{\hat{R}} = \cX \times_{\cY} \Spec \hat{R}$, and
  $(\cX_0)_{\hat{R}}$ is the preimage of $\cX_0$ under
  $\cX_{\hat{R}} \to \cX$.
\end{definition}
We have the following simple lemma.
\begin{lemma} \label{L:comparison-relative-formal} Let
  $f\colon \cX \to \cY$ be a morphism of noetherian algebraic stacks,
  $\cY_0 \subset \cY$ a closed substack, and
  $\cX_0 \subset f^{-1}(\cY_0)$ a closed substack. If the triple
  $(f,\cX_0,\cY_0)$ satisfies relative formal functions, then for
  every $\cF \in D^+_{\coh}(\cX)$ the natural comparison morphism:
  \[
    (\RDERF f_\ast \cF)^{\wedge}
    \to \RDERF\hat{f}_\ast \hat{\cF}
  \]
  is an isomorphism. If, in addition, the pair $(\cY, \cY_0)$
  satisfies formal functions, then so does $(\cX, \cX_0)$.
\end{lemma}
\begin{proof}
  The second claim follows from the first since for any coherent sheaf
  $\cF$ on $\cX$, we have isomorphisms:
  \[
    \rgamma(\cX, \cF) \simeq \rgamma(\cY, \RDERF f_\ast \cF) \simeq
    \rgamma(\widehat{\cY}, (\RDERF f_\ast \cF)^{\wedge}) \simeq
    \rgamma( \widehat{\cY}, \RDERF\widehat{f}_\ast \hat{\cF}) \simeq
    \rgamma(\widehat{\cX}, \widehat{\cF}).
  \]
  To treat the first claim, let $\Spec R \to \cY$ be a smooth
  morphism, $I \subset R$ the ideal defined by the preimage of
  $\cY_0 \subset \cY$, and $\widehat{R}$ the $I$-adic completion of
  $R$. Let $f_R \colon \cX_R \to \Spec R$ be the base change morphism,
  where $\cX_R = \cX \times_\cY \Spec R$. For an integer $i$,
  $\mathscr{H}^i((\RDERF f_\ast \cF)^{\wedge})$ is the sheafification
  of the presheaf:
  \[
    (\spec R \to \cY) \mapsto \mathrm{H}^i(\cX_R, \cF_R) \tensor_R
    \widehat{R} \simeq \mathrm{H}^i(\cX_{\hat{R}},\cF_{\hat{R}}),
  \]
  with the last isomorphism by flat base change. One similarly finds
  that $\mathscr{H}^i(\RDERF\widehat{f}_\ast \hat{\cF})$ is the
  sheafification of the presheaf:
  \[
    (\spec R \to \cY) \mapsto \mathrm{H}^i(\widehat{\cX}_{\hat{R}},
    \hat{\cF}_{{\hat{R}}} ).
  \]
  By assumption, the pair $(\cX_{\hat{R}},(\cX_0)_{\hat{R}})$
  satisfies formal functions, and the result follows.
\end{proof}
The following is a reformulation of \Cref{T:strong-cp-implies-ff} in
the relative and derived situation.
\begin{theorem}\label{T:strong-cp-comparison}
  Let $f\colon \cX \to \cY$ be a finite type morphism of noetherian
  algebraic stacks, $\cY_0 \subset \cY$ a closed substack, and
  $\cX_0 \subset f^{-1}(\cY_0)$ a closed substack. Assume that the
  triple $(f,\cX_0,\cY_0)$ is strongly cohomologically proper.
  \begin{enumerate}[{\rm(1)}]
  \item \label{TI:comparison:ff} The triple $(f,\cX_0,\cY_0)$
    satisfies relative formal functions.
  \item \label{TI:comparison:compare} The comparison morphism
    $(\RDERF f_* \cF)^\wedge \to \RDERF \hat{f}_*\hat{\cF}$ is an
    isomorphism for all $\cF \in D^+_{\Coh}(\cX)$.
  \item \label{TI:comparison:formal} The functor $\RDERF \hat{f}_*$
    sends $D^+_{\Coh}(\hat{\cX})$ to $D^+_{\Coh}(\hat{\cY})$.
  \end{enumerate}
\end{theorem}
\begin{proof}
  By \Cref{L:comparison-relative-formal},
  \eqref{TI:comparison:ff}$\Rightarrow$\eqref{TI:comparison:compare}. The
  other claims follow from \Cref{T:strong-cp-implies-ff}.
\end{proof}

Cohomologically proper morphisms and relative formal functions are
stable under composition. The analogous question for coherent
completeness seems more subtle.  The following theorem asserts that
they ascend under proper representable morphisms, however.

\begin{theorem}\label{thm:properness-ascent}
  Let $\cY$ be a noetherian algebraic stack with affine stabilizers and
  $\cY_0 \subset \cY$ a closed substack. Let $f \colon \cX \to \cY$ be
  a proper and representable morphism. If $(\cY,\cY_0)$ satisfies
  formal functions and is coherently complete, then so does $(\cX,f^{-1}(\cY_0))$.
\end{theorem}

\begin{remark} \label{rmk:hlp-properness-ascent}
    This theorem is similar in spirit to the statement of \cite[Thm.~4.2.1]{ref:hlp} that cohomologically projective morphisms (see \cite[Def.4.1.3]{ref:hlp}) are formally proper.  Projective morphisms are cohomologically projective and thus formally proper.  Using Rydh's Chow Lemma \cite{ref:rydh-equivariant} and the fact that formal properness descends under proper surjective morphisms (\Cref{rmk:formal-properness-descends}), it follows that proper morphisms are also formally proper.
\end{remark}

\begin{proof}[Proof of \Cref{thm:properness-ascent}]
  Let $\cX_0 = f^{-1}(\cY_0)$.   We have the following diagram:
  \[
    \xymatrix@C+2pc{D^+_{\qcoh}(\cX) \ar[r] \ar[d]_-{\RDERF f_*} & D_{\qcoh}(\cX) \ar@<+0.5ex>[r]^-{c^*} & \ar@<+0.5ex>[l]^-{c_{{\qcoh},*}} D(\hat{\cX}) & \ar[l] D^+(\hat{\cX}) \ar[d]^{\RDERF \hat{f}_*}\\
      D^+_{\qcoh}(\cY) \ar[r] & D_{\qcoh}(\cY) \ar@<+0.5ex>[r]^-{d^*} &
      \ar@<+0.5ex>[l]^-{d_{{\qcoh},*}} D(\hat{\cY}) & \ar[l]
      D^+(\hat{\cY}).}
  \]
  Since $f$ is proper and representable, it is universally
  cohomologically proper \cite[Thm.\ 1.2]{ref:olsson}, so the
  triple $(f,\cX_0,\cY_0)$ is strongly cohomologically proper
  (\Cref{E:adic-strong-cp}). In particular,
  \Cref{T:strong-cp-comparison} implies that $f$ satisfies relative
  formal functions and $\RDERF \hat{f}_*$ sends
  $D^+_{\coh}(\hat{\cX})$ to $D^+_{\coh}(\hat{\cY})$. A simple
  calculation using homotopy limits shows that because $f$ is
  representable, $\RDERF \hat{f}_*$ even sends $D^b_{\coh}(\hat{\cX})$
  to $D^b_{\coh}(\hat{\cY})$. In particular,
  \Cref{L:comparison-relative-formal} implies that $(\cX,\cX_0)$
  satisfies formal functions. By
  \Cref{P:underived-derived-comparison}\eqref{PI:underived-derived-comparison:cc}
  it remains to prove that \[\Coh(\cX) \to \Coh(\hat{\cX})\] is
  essentially surjective.  We first prove this under the following
  assumptions:
  \begin{enumerate}
  \item $f$ is projective (so comes with a relatively ample line bundle $\mathcal{L}$); and 
  \item $\cX_0$ is a Cartier divisor. 
  \end{enumerate}
  In this case, let $\mathfrak{F} \in \Coh(\hat{\cX})$ and let $n\in \bZ$. Then by the projection formula \cite[Thm.~A.12]{ref:hall}, 
  \begin{align*}
    \RDERF f_*( \mathcal{L}^{\otimes n} \otimes^{\LDERF}_{\oh_{\cX}} c_{{\qcoh},*}\mathfrak{F}) &\simeq     \RDERF f_*c_{{\qcoh},*}( {\hat{\mathcal{L}}}^{\otimes n} \otimes^{\LDERF}_{\oh_{\hat{\cX}}} \mathfrak{F}) && \\
    &\simeq d_{{\qcoh},*}\RDERF \hat{f}_*( {\hat{\mathcal{L}}}^{\otimes n} \otimes^{\LDERF}_{\oh_{\hat{\cX}}} \mathfrak{F}).
  \end{align*}
  But
  $\RDERF \hat{f}_*( {\hat{\mathcal{L}}}^{\otimes n}
  \otimes^{\LDERF}_{\oh_{\hat{\cX}}} \mathfrak{F}) \in
  D^b_{\coh}(\hat{\cY})$ and since $(\cY,\cY_0)$ is derived coherently
  complete, by  (\Cref{P:underived-derived-comparison}\eqref{PI:underived-derived-comparison:cc})
$d_{{\qcoh},*}\RDERF \hat{f}_*( {\hat{\mathcal{L}}}^{\otimes n}
  \otimes^{\LDERF}_{\oh_{\hat{\cX}}} \mathfrak{F}) \in D^b_{\coh}(\cY).$

  Let $p\colon \Spec A \to \cY$ be a smooth cover by an affine
  scheme. Let $\cX_A = \cX\times_{\cY} \Spec A$ and let
  $f_A \colon \cX_A \to \Spec A$ and $q \colon \cX_A \to \cX$ be the
  induced morphisms. We may choose $p$ so that $q^*\mathcal{L}$ is ample relative to $f_A$. Then
  \begin{align*}
    \RDERF \Gamma(\cX_A, (q^*\mathcal{L})^{\otimes n} \otimes^{\LDERF}_{\oh_{\cX_A}} q^*c_{{\qcoh},*}\mathfrak{F}) &\simeq \RDERF \Gamma(\cX_A, q^*(\mathcal{L}^{\otimes n} \otimes^{\LDERF}_{\oh_{\cX}} c_{{\qcoh},*}\mathfrak{F})) \\
                                                                                                       &\simeq \RDERF \Gamma(\Spec A, \RDERF f_{A,*}q^*(\mathcal{L}^{\otimes n} \otimes^{\LDERF}_{\oh_{\cX}} c_{{\qcoh},*}\mathfrak{F}))\\
                                                                                                       &\simeq \RDERF \Gamma(\Spec A, p^*\RDERF f_*(\mathcal{L}^{\otimes n} \otimes^{\LDERF}_{\oh_{\cX}} c_{{\qcoh},*}\mathfrak{F}))\\
    &\in D^b_{\coh}(A).
  \end{align*}
  Since this is true for all $n \in \bZ$ and $q^*\mathcal{L}$ is ample relative to
  $f_A$, it follows from \cite[Thm.\ 3.8]{ref:hall} that
  $q^*c_{{\qcoh},*}\mathfrak{F} \in D^b_{\coh}(\cX_A)$. By smooth
  descent, $c_{{\qcoh},*}\mathfrak{F} \in D^b_{\coh}(\cX)$. It remains
  to prove that the adjunction morphism
  \begin{eqnarray} \label{eq:adjunction_iso}
  c^*c_{{\qcoh},*}\mathfrak{F} \to \mathfrak{F} 
  \end{eqnarray}
  is an isomorphism.
  
To prove this, observe that $\oh_{\cX_0}$ is a perfect complex on $\cX$ since $\cX_0 \subseteq \cX$ is Cartier. It follows from the
  projection formula that
  \[
    (c^*c_{{\qcoh},*}\mathfrak{F})\otimes^{\LDERF}_{\oh_{\hat{\cX}}}
    c^*\oh_{\cX_0} \simeq
    c^*(c_{{\qcoh},*}\mathfrak{F}\otimes^{\LDERF}_{\oh_{{\cX}}}
    \oh_{\cX_0}) \simeq c^*c_{{\qcoh},*}(\mathfrak{F}
    \otimes^{\LDERF}_{\oh_{\hat{\cX}}} c^*\oh_{\cX_0}) \simeq \mathfrak{F}
    \otimes^{\LDERF}_{\oh_{\hat{\cX}}} c^*\oh_{\cX_0}.
  \]
In summary, if $\mathfrak{K}$ denotes the cone of (\ref{eq:adjunction_iso}), we have shown that
  $\mathfrak{K} \otimes^{\LDERF}_{\oh_{\hat{\cX}}} c^*\oh_{\cX_0}
  \simeq 0$. To finish the proof, suppose that $\mathfrak{K} \neq 0$. Then since $\mathfrak{K} \in D^b_{\coh}(\widehat{\cX})$, we may choose a largest integer $q$ such that $\mathscr{H}^q(\mathfrak{K}) \neq 0$. Then
  \[ \mathscr{H}^q(\mathfrak{K}) \otimes_{\oh_{\hat{\cX}}} c^*\oh_{\cX_0} \simeq 
\mathscr{H}^q(\mathfrak{K} \otimes^{\LDERF}_{\oh_{\hat{\cX}}} c^*\oh_{\cX_0} )\simeq 0\]
and so by Nakayama's Lemma, $\mathscr{H}^q(\mathfrak{K}) \simeq 0$, which is a contradiction.

In general, Rydh's Chow Lemma \cite{ref:rydh-equivariant} provides a
blow-up $\cX' \xrightarrow{\pi} \cX$ such that the composition
$g=f\circ \pi$ is projective. We may replace $\cX'$ by an additional
blow-up so that $\cX_0'=g^{-1}(\cY_0) \subseteq \cX'$ becomes
Cartier. It follows from the case already considered that
$(\cX',\cX'_0)$ is derived coherently complete. By
\Cref{thm:descent}\eqref{thm:descent3}, it follows that $(\cX,\cX_0)$
is derived coherently complete.
\end{proof}

We also have the following easy result.

\begin{lemma} \label{lem:permanence-under-larger-substacks}
  Let $\cX$ be a noetherian algebraic stack and $\cX_0 \subset \cX$ a closed substack.  Suppose that $\cZ \subset \cX$ is a closed substack such that $|\cX_0| \subseteq |\cZ|$.  If $(\cX, \cX_0)$ satisfies formal functions (resp. is coherently complete), then the same holds for $(\cX, \cZ)$.  
\end{lemma}

We collect here some examples of algebraic stacks $\cX$ that satisfy each of the three properties: cohomological properness, formal functions and coherent completeness.

\begin{example} [Noetherian affine schemes]
  Let $A$ be a noetherian ring and let $I \subseteq A$ be an
  ideal. Then 
  $(\Spec A,\Spec A/I)$ is coherently complete if
  and only if $A$ is $I$-adically complete \cite[Ex. 3.9]{ref:ahr2}. 
\end{example}

\begin{example}[Proper algebraic stacks] \label{ex:proper} A proper
  morphism $\cX \to \cY$ of noetherian algebraic stacks is universally
  cohomologically proper \cite[Thm.~1.2]{ref:olsson}. In the case that
  $\cY = \Spec R$ and $I \subset R$ is an ideal such that $R$ is
  $I$-adically complete, then setting $\cX_0 \subset \cX$ to be the
  preimage of $\Spec R/I \subset \Spec R$, the pair $(\cX, \cX_0)$
  satisfies formal functions and is coherently complete
  \cite{ref:olsson,ref:conrad}. 
\end{example} 

\begin{example}[Good moduli spaces] \label{ex:good} Let
  $\cX \to \spec R$ be a good moduli space \cite{ref:alper}, where
  $\cX$ is a noetherian algebraic stack with affine diagonal. Then
  $\cX \to \spec R$ is of finite type \cite[Thm.~A.1]{ref:ahr}, and
  $R$ is noetherian and $\cX$ is cohomologically proper over $R$
  \cite[Thm.~4.16 (x)]{ref:alper}. Since good moduli spaces are
  compatible with arbitrary base change, $\cX$ is even universally
  cohomologically proper over
  $R$. 
  
  Let $\cX_0 \subset \cX$ be a closed substack defined by a coherent
  sheaf of ideals $\mathcal{I} \subset \cO_\cX$ and let
  $I = \Gamma(\cX, \mathcal{I})$ to be the corresponding ideal of
  $R$. If $\cX_0$ has the resolution property and $R$ is $I$-adically
  complete, then the pair $(\cX, \cX_0)$ is coherently complete and
  satisfies formal functions \cite[Thm.~1.6]{ref:ahr2}.  When $\cX$
  has the resolution property and $\cX_0$ is the preimage of
  $\Spec R/I$, then this was the main result of \cite{ref:gerdzb}. If
  $R$ is equicharacteristic, $\cX$ has the resolution property and
  $\cX_0$ is supported at a closed point, then this is
  \cite[Thm.~1.3]{ref:ahr}.

  We also point out the following converse. Let $\cX$ be a noetherian
  algebraic stack with noetherian adequate moduli space
  $\pi \colon \cX \to X$. Let $\cX_0 \subseteq \cX$ be a closed
  substack with scheme-theoretic image $X_0 \subset X$.  If
  $X=\Spec A$ is affine and the pair $(\cX,\cX_0)$ is coherently
  complete, then the pair $(X,X_0)$ is coherently complete
  \cite[Lem.~3.5(1)]{ref:ahr2}.
\end{example}

\begin{example}[Adequate moduli spaces] 
  Let $\cX = [\Spec A/G]$, where $G$ is a reductive algebraic group
  over a noetherian ring $R$ acting on an affine scheme $\Spec A$ of
  finite type over $R$.  Then the cohomology ring
  $H^*(\cX, \oh_{\cX}) = H^*(G,A)$ is a finitely generated
  $R$-algebra.  See \cite[Thm.\ 1.1]{touze-vanderkallen} for the case
  when $R$ is a field and \cite{MR3525842} in general. It follows that
  $\cX \to \Spec (A^G)$ is cohomologically proper.  If $J \subseteq A$
  is a $G$-equivariant ideal $A^G$ is $J^G$-adically complete, then
  $(\cX,[\Spec(A/J)/G])$ satisfies formal functions
  (\Cref{cor:reductive-ff}).
\end{example}

\begin{example}[Unipotent groups] \label{ex:unipotent}
  If $U$ is a unipotent (affine) algebraic group over an algebraically closed field $k$, then $BU$ is cohomlogically proper if and only if ${\rm char}(k) = 0$.  Indeed, if ${\rm char}(k) = 0$, then $U$ admits a filtration by $\bG_a$'s and thus it suffices to show that $B \bG_a$ is cohomologically proper.  Any finite dimensional representation of $\bG_a$ has a filtration by trivial representations and so it is enough to compute the cohomology $\mathrm{H}^i(B \bG_a, k)$ of the trivial representation.  Using the \v{C}ech complex corresponding to the cover $\Spec k \to B \bG_a$, one can compute that $\mathrm{H}^i(B \bG_a, k) = k$ if $i=0,1$ and zero otherwise.   

  On the other hand, if ${\rm char}(k) = p$, then we can compute (again using the \v{C}ech complex) that $H^1(B \bG_a, k)$ is infinite dimensional.  Thus $B \bG_a$ is not cohomologically proper and it follows that neither is $BU$ for any unipotent group $U$.

  Supposing again that ${\rm char}(k)=0$, let $(R, \fm)$ be a complete local noetherian ring with residue field $k$ such that $R$ is not artinian.  
  Then $(BU_R, BU_{k})$ is not coherently complete.  It suffices to assume that $U = \bG_a$ and that $R = k[[\epsilon]]$.  Consider the exponential map $\bG_{a, R/\fm^n} \to \bG_{m, R/\fm^n}$ defined by $t \mapsto \exp(\epsilon x) = 1 + (\epsilon x) + \frac{1}{2} (\epsilon x)^2 + \cdots $, where $t$ and $x$ are coordinates of $\bG_m$ and $\bG_a$, respectively.  This defines a compatible system of non-trivial line bundles on $B \bG_{a, R/\fm^n}$ that does not algebraize to a line bundle on $B \bG_{a,R}$. Indeed, every line bundle on $B \bG_{a,R}$ is trivial. 
\end{example}

\begin{example}[Universal cohomological properness of $BG$ in characteristic
  0] \label{ex:cohomological-properness-BG-char0} If $G$ is an affine group scheme of finite
  type over an algebraically closed field $k$ of characteristic $0$,
  then $BG$ is cohomologically proper.  To see this, let $G_u$ be the
  unipotent radical of $G$ so that $G_r:=G/G_u$ is linearly reductive.
  Since $BG_r$ has vanishing higher coherent cohomology, it is
  cohomologically proper.  Therefore, it suffices to prove that
  $BG \to BG_r$ is cohomologically proper. By flat base
  change and descent, it remains to prove that the base change
  $BG_u \to \Spec k$ (of $BG \to BG_r$ by $\Spec k \to BG_r$) is
  cohomologically proper but this follows from \Cref{ex:unipotent}.

  More generally, $BG \to \Spec k$ is universally cohomologically
  proper. In particular, it follows that if $T$ is an $J$-adically
  complete noetherian $k$-algebra, then $(BG_T, BG_{T/J})$ satisfies
  formal functions (\Cref{cor:adic-cp-ff}). To see this, let $R$ be a
  noetherian $k$-algebra.  We first observe that $BG_R$ has finite
  cohomological dimension.  Since $G$ is affine, we may choose an
  embedding $G \subset \GL_n$.  This induces a morphism
  $BG_R \to B \GL_{n,R}$ of algebraic stacks such that the base change
  by $\Spec R \to B \GL_{n,R}$ is the quotient $\GL_{n,R} / G_R$,
  which we know has finite cohomological dimension.  Let $d$ be the
  cohomological dimension.  If $M$ is a coherent $BG_R$-module, then
  there is a $BG_k$-module $N$ and a surjection
  $N_R=N \tensor_k R \to M$.
  Indeed, the morphism $r\colon BG_R \to BG_k$ is affine, so the adjunction $r^*r_*M \to M$ is surjective. We may write $r_*M = \varinjlim_\lambda N_\lambda$, where each $M_\lambda$ is a vector bundles on $BG_k$. Since $M$ is coherent on $BG_R$, there is a $\lambda_0$ sufficiently large such that $r^*r_*N_\lambda \twoheadrightarrow M$ for all $\lambda\geq \lambda_0$. Now take $N=N_{\lambda_0}$.
  Letting $L$ be its kernel, there is a long exact sequence
  \[
    \cdots \to \mathrm{H}^d(BG_R,L) \to \mathrm{H}^d(BG_R, N_R) \to \mathrm{H}^d(BG_R,M) \to 0
  \]
  By flat base change, we know that
  $\mathrm{H}^d(BG_R,N_R) = \mathrm{H}^d(BG,N) \tensor_k R$, which is a finite
  $R$-module, so $\mathrm{H}^d(BG_R,M)$ is a finite $R$-module.  Since this
  holds for any coherent $BG_R$-module $M$, we may perform descending
  induction to conclude that $\mathrm{H}^i(BG_R,M)$ is a finite
  $R$-module for all $i$.
\end{example}

\section{Descent of properties}

The main result of this section is a criterion for descending cohomological properness, formal functions and coherent completeness along universally submersive morphisms.  Recall that a finite type morphism $\cX \to \cY$ of noetherian algebraic stacks is \emph{universally submersive} if for every morphism $V \to \cY$ from a scheme, the base change $\cX \times_{\cY} V \to V$ is surjective and $V$ has the quotient topology.  This is equivalent to requiring that for every morphism $\Spec R \to \cY$ from a DVR $R$ there is an extension of DVRs $R \to R'$ and a lift $\Spec R' \to \cX$ of the composition $\Spec R' \to \Spec R \to \cX$.

The most important examples for us are when $\cX \to \cY$ is representable, proper, and surjective, or when $\cX \to \cY$ is representable, faithfully flat, and locally of finite presentation.  For proper morphisms, these descent results are rather standard. The novelty here is that they hold more generally as long as you require the corresponding property on not only $\cX$ but the higher fiber products
\[ (\cX/\cY)^n \coloneqq \underbrace{\cX \times_{\cY} \cdots \times_{\cY} \cX}_{n} \]
for all $n \ge 1$.

\begin{theorem} \label{thm:descent} Let $f \co \cX \to \cY$ be a
  representable, universally submersive and finite type morphism of
  noetherian algebraic
  stacks.  
  \begin{enumerate}[{\rm(1)}]
    \item \label{thm:descent1} 
      Suppose that $\cY$ is of finite type over a noetherian ring $R$.  If $(\cX/\cY)^n$ is cohomologically proper over $\Spec R$ for $n \ge 1$, then so is $\cY$.
    \item \label{thm:descent2}
      Let $\cY_0 \subset \cY$ be a closed substack and $\cX_0 \subset \cX$ be its preimage.  If $((\cX/\cY)^n, (\cX_0/\cY_0)^n)$ satisfies formal functions for $n \ge 1$, then so does $(\cY, \cY_0)$.
    \item \label{thm:descent3} 
      Let $\cY_0 \subset \cY$ be a closed substack defined by a coherent sheaf of ideals $\cI \subset \oh_{\cY}$ and let $\cX_0 \subset \cX$ be its preimage.
If $(\cX, \cX_0)$ is coherently complete and $((\cX/\cY)^n, (\cX_0/\cY_0)^n)$ satisfies formal functions for $n \ge 1$, then $(\cY, \cY_0)$ is coherently complete and satisfies formal functions.
  \end{enumerate}
  If $f \co \cX \to \cY$ is proper and surjective, then in each case one only needs to require the properties for $n=1$.
\end{theorem}

\begin{remark} \label{rmk:formal-properness-descends}
  This theorem is similar in spirit to the statement \cite[Thm.~3.3.1]{ref:hlp} that an h-covering $\cX' \to \cX$ induces an equivalence of $\infty$-categories of almost perfect complexes on $\cX$ and on the \v{C}ech nerve of $\cX' \to \cX$ and its consequence \cite[Thm.~4.3.1]{ref:hlp} that formal properness descends under proper surjective morphisms.
\end{remark}

\begin{proof}
  We will argue \eqref{thm:descent1} and \eqref{thm:descent2} by noetherian induction 
on the abelian category $\Coh(\cY)$ and we may assume that every proper closed substack $\cZ \subsetneq \cY$ is cohomologically proper over $R$ and that the pair $(\cZ, \cZ \cap \cX_0)$ satisfies formal functions.  For \eqref{thm:descent3}, note that $(\cY, \cY_0)$ satisfies formal functions by \eqref{thm:descent2} and the functor $\Coh(\cY) \to \Coh(\widehat{\cY})$
  is fully faithful, with image stable under kernels, cokernels and extensions, by \Cref{C:fullyfaithful}. Thus it suffices to show that $\Coh(\cY) \to \Coh(\widehat{\cY})$ is essentially surjective.  To accomplish this, we will also argue via d\'evissage on the abelian category $\Coh(\widehat{\cY})$ and we may assume that if $\fG$ is a coherent sheaf on $\hat{\cY}$ annihilated by $\cJ \cO_{\widehat{\cY}}$ for some non-zero coherent sheaf of ideals $\cJ \subset \cO_\cY$, then $\fG$ is in the essential image. 
  
  In each case, we may further reduce to the situation where $\cY$ is reduced. By generic flatness, there is a dense open  substack $\cU \subset \cY$ such that $f^{-1}(\cU) \to \cU$ is flat. By Rydh's  extension \cite{ref:rydh} of Raynaud-Gruson's theorem \cite{ref:rayanud-gruson} on flatification by blow-ups, there is a commutative diagram:
 \[ \begin{tikzcd}
  \cX' \ar{r}{g'}\ar{d}{f'}  & \cX \ar{d}{f}  \\
  \cY' \ar{r}{g} & \cY,
\end{tikzcd} \]
where $g \co \cY' \to \cY$ is a blow-up along a closed substack contained in the complement of $\cU$, $\cX'$ is the strict transform of $\cX$ along $g \co \cY' \to \cY$, and  $f' \co \cX' \to \cY'$ is flat.  Observe that $g \co \cY' \to \cY$ is a surjective, proper and representable morphism.  Moreover, since $f \co \cX \to \cY$ is universally submersive, 
$f' \co \cX' \to \cY'$ is surjective.  Indeed, for every point $y' \in |\cY'|$ there is a morphism $\Spec R \to \cY'$ from a DVR whose generic point maps to $g^{-1}(\cU)$.  As $\cX \to \cY$ is universally submersive, there exists an extension of DVRs $R \to R'$ and a lift $\Spec R' \to \cX$ of $\Spec R' \to \Spec R \to \cY' \to \cY$.  The induced map $\Spec R' \to \cX \times_{\cY} \cY'$ factors through the strict transform $\cX'$.  The image of the closed point under $\Spec R' \to \cX'$ is a preimage of $y'$ and we conclude that  $f' \co \cX' \to \cY'$ is a faithfully flat morphism of finite type.

 Since $g \co \cY' \to \cY$ is separated and representable, its diagonal is a closed immersion and it follows that $(\cX'/\cY')^n \to (\cX'/\cY)^n$ is also closed immersion for each $n \geq 1$. Since $g' \co \cX' \to \cX$ is proper, $(\cX'/\cY)^n \to (\cX/\cY)^n$ is proper and it follows that $(\cX'/\cY')^n \to (\cX/\cY)^n$ is proper and thus universally cohomologically proper (\Cref{ex:proper}).

 We now establish \eqref{thm:descent1}. 
 The hypotheses that each $(\cX/\cY)^n$ is cohomologically proper over $R$ implies that $(\cX'/\cY')^n$ is also cohomologically proper over $R$.
Since $f' : \cX' \to \cY'$ is faithfully flat, cohomological descent implies that for a coherent sheaf $\cM$ on $\cY'$, there is a convergent spectral sequence:
\[ 
  \mathrm{H}^i((\cX'/\cY')^j, \cM_{(\cX'/\cY')^j}) 
  \Rightarrow \mathrm{H}^{i+j}(\cY', \cM)
\]
where $\cM_{(\cX'/\cY')^j}$ denotes the pullback of $\cM$ under the projection $(\cX'/\cY')^j \to \cY'$. 
It follows immediately that $\cY'$ is cohomologically proper over $R$.

For a coherent sheaf $\cF$ on $\cY$, consider the exact triangle of complexes on $\cY$:
\[
  \cF \to \RDERF g_\ast g^\ast \cF \to \cC \to \cF[1]. 
\]
Since $g$ is proper and representable, $\mathrm{R}g_\ast g^\ast \cF$ and also $\cC$ have bounded coherent cohomology.  Since $\cY'$ is cohomologically proper over $R$, we conclude (using \Cref{rmk:cp-equivalence}) that 
\[ 
  \rgamma(\cY,  \RDERF g_\ast g^\ast \cF ) 
  \simeq \rgamma(\cY', g^\ast \cF) \in D^+_{\coh}(R).
\]
Since $g \co \cY' \to \cY$ is an isomorphism over $\cU \subset \cY$, $\cC|_U \simeq 0$.  Therefore in order to show that $\rgamma(\cY, \cF) \in  D^+_{\coh}(R)$, it suffices to show that $\rgamma(\cY, \cC) \in D^+_{\coh}(R)$.  By using the spectral sequence $\RDERF^i \pi_\ast \mathscr{H}^j(\cC) \Rightarrow \mathscr{H}^{i+j}(\RDERF\pi_\ast \cC)$ as in \Cref{rmk:cp-equivalence}, it suffices to show that if $\cE$ is a coherent sheaf on $\cY$ such that $\cE|_U = 0$, then $\rgamma(\cY, \cE) \in D^+_{\coh}(R)$.  But in this case $\cE$ is supported on a proper closed substack $\cZ \subsetneq \cX$ and the statement follows from the d\'evissage hypothesis.  

For \eqref{thm:descent2} and \eqref{thm:descent3}, we also establish
the notation that the closed substacks $\cY'_0 \subset \cY'$ and
$\cX'_0 \subset \cX'$ denote the preimages of $\cY_0$.  For
\eqref{thm:descent2}, since $(\cX'/\cY')^n \to (\cX/\cY)^n$ is
universally cohomologically proper and
$((\cX/\cY)^n, (\cX_0/\cY_0)^n)$ satisfies formal functions for each
$n \ge 1$, it follows from \Cref{L:comparison-relative-formal},
\Cref{T:strong-cp-comparison}\eqref{TI:comparison:ff} and
\Cref{E:adic-strong-cp} that $((\cX'/\cY')^n, (\cX'_0/\cY'_0)^n)$ also
satisfies formal functions.

Since $f'\co \cX' \to \cY'$ is faithfully flat, cohomological descent implies that for each coherent sheaf $\cM'$ on $\cY'$, we have a morphism of convergent spectral sequences:
\[ 
  \begin{tikzcd} 
    \mathrm{H}^i((\cX'/\cY')^j, \cM'_{(\cX'/\cY')^j}) \ar[Rightarrow]{r} \ar{d}
      & \mathrm{H}^{i+j}(\cY', \cM') \ar{d} \\
    \mathrm{H}^i((\cX'/\cY')^{j \wedge{\hspace{1mm}}}, \widehat{\cM'}_{(\cX'/\cY')^{j \wedge{\hspace{1mm}}}}) \ar[Rightarrow]{r} 
      &  \mathrm{H}^{i+j}(\widehat{\cY'}, \widehat{\cM'}), 
  \end{tikzcd} 
\]
where $(\cX'/\cY')^{j \wedge{\hspace{1mm}}}$ is the completion of
$(\cX'/\cY')^j$ along $(\cX'_0/\cY'_0)^j$ and $\hat{\cY'}$ is the
completion of $\cY'$ along $\cY'_0$.  Since
$((\cX'/\cY')^n, (\cX'_0/\cY'_0)^n)$ satisfies formal functions, so
does $(\cY', \cY'_0)$.

For a coherent sheaf $\cF$ on $\cY$, we again consider the exact triangle:
\[
  \cF \to \RDERF g_\ast g^\ast \cF \to \cC \to \cF[1],
\]
which induces a morphism of exact triangles 
\[ \begin{tikzcd} 
  \rgamma(\cY,\cF) \ar{r} \ar{d} 
    & \rgamma(\cY, \RDERF g_\ast g^\ast \cF )\ar{r} \ar{d} 
    &  \rgamma(\cY, \cC)  \ar{r} \ar{d}
    &  \rgamma(\cY, \cF)[1] \ar{d} \\
 \rgamma(\widehat{\cY}, \widehat{\cF}) \ar{r} 
  & \rgamma(\widehat{\cY}, (\RDERF g_\ast g^\ast \cF)^{\wedge{\hspace{1mm}}} )\ar{r} 
  &  \rgamma(\widehat{\cY}, \widehat{\cC}) \ar{r} 
  &  \rgamma(\widehat{\cY}, \widehat{\cF})[1].
 \end{tikzcd} \]

Since $(\cY', \cY'_0)$ satisfies formal functions, the second vertical map is an isomorphism.   Since $\cC|_U \simeq 0$, the third vertical map is also a isomorphism.  Indeed, we can reduce as above to the case that $\cC$ is a coherent sheaf supported in cohomological degree 0 in which case $\cC$ is supported on a proper closed substack $\cZ \subsetneq \cY$ and we may apply the d\'evissage hypothesis. Thus $\rgamma(\cY, \cF) \to \rgamma(\hat{\cY}, \hat{\cF})$ is a an isomorphism and $(\cY, \cY_0)$ satisfies formal functions.

For \eqref{thm:descent3}, let $\fG$ be a coherent sheaf on $\hat{\cY}$.  By coherent completeness of $(\cX, \cX_0)$, there is a coherent sheaf $\cF$ on $\cX$ and an isomorphism $\hat{\cF} \cong \hat{f}^* \fG$.  Denoting by $\cF'$ the pullback $g'^* \cF$, we have that completion $\hat{\cF'}$ of $\cF'$ along $\cX'_0$ is identified with $\hat{g'}^* \hat{f}^* \fG \cong \hat{f'}^* \hat{g}^{\, *} \fG$.  Letting $p_1$ and $p_2$ denote the two projections $\cX' \times_{\cY'} \cX'$, then because $\hat{\cF'}$ is the pullback of a coherent sheaf on $\hat{\cY'}$ we have an isomorphism
\[
  \vartheta : \hat{ p^*_1 \cF'} 
  \iso \hat{ p^*_2 \cF'}
\]
of coherent sheaves on the completion of $(\cX'/\cY')^2$ along the preimage of $\cY_0$
satisfying the cocycle condition on the completion of $(\cX'/\cY')^3$ along the preimage of $\cY_0$. 
As the morphisms $(\cX'/\cY')^n \to (\cX/\cY)^n$ are universally cohomologically proper for $n \ge 1$,  $((\cX'/\cY')^n, (\cX'_0/\cY'_0)^n)$ satisfy formal functions.
It  follows that there is a unique isomorphism $\theta \co p_1^* \cF' \iso p_2^* \cF'$ of coherent sheaves on $(\cX'/\cY')^2$ satisfying the cocycle condition on $(\cX'/\cY')^3$ such that $\hat{\theta} = \vartheta$.
By faithfully flat descent, there is a coherent sheaf $\cG'$ on $\cY'$ such that there is an isomorphism $f'^* \cG' \cong \cF'$ compatible with $\theta$. Since the coherent sheaves $\hat{\cG'}$ and $\hat{g}^* \fG$ are described via the same formal descent data with respect to $\hat{\cX'} \to \hat{\cY'}$, there is an isomorphism $\hat{ \cG'} \cong \hat{g}^* \fG$.  \Cref{T:strong-cp-comparison}\eqref{TI:comparison:compare} now gives equivalences
$\hat{g}_* \hat{g}^* \fG \cong \hat{g}_* \hat{\cG'} \cong \hat{g_* \cG'}$ 
and thus $\hat{g}_* \hat{g}^* \fG$ is in the essential image of $\Coh(\cY) \to \Coh(\hat{\cY})$.
As $g \co \cY' \to \cY$ is an isomorphism over $\cU$, the adjunction morphism
$$\fG \to \hat{g}_* \hat{g}^* \fG$$
is an isomorphism over $\cU$.  We claim that the kernel and cokernel of this adjunction morphism are annihilated by some power of sheaf ideals $\cJ \subset \oh_{\hat{\cY}}$ defining the reduced complement of $\cU$.  Indeed, this can be checked on a smooth presentation $\Spec B \to \cY$.  Let $I \subset B$ be the ideal defined by the preimage of $\cY_0$ in $\Spec B$,  let $\hat{B}$ be the $I$-adic completion of $B$, and let $\cY'_B = \cY' \times_{\cY} \Spec B$ be endowed with the (proper) projection morphism $g_B \co \cY'_B \to \Spec B$.  The module of sections of $\fG$ over $\Spec B \to \cY$ is a finite $\hat{B}$-module $N$ while the module of sections of $\hat{g}_* \hat{g}^* \fG$ over $\Spec B \to \cY$ is $\mathrm{H}^0(\hat{\cY'}_B, \hat{g_B}^* N)$, where $\hat{g_B}$ is the completion of $g$ along the preimage of $\cY_0$.  Moreover, the adjunction morphism $\fG \to \hat{g}_* \hat{g}^* \fG$ corresponds to the morphism 
$$N \to \mathrm{H}^0(\hat{\cY'}_B, \hat{g_B}^* N) \simeq \mathrm{H}^0(\cY'_{\hat{B}},g_{\hat{B}}^*N)$$
and both the kernel and cokernel are annihilated by some power of the ideal of sections of $\cJ$ over $\Spec B \to \cY$ because $\cY'_{\hat{B}} \to \Spec \hat{B}$ is an isomorphism over the open subset $\Spec \hat{B} - \Spec (\hat{B}/\mathscr{J}(\Spec B)\hat{B})$.  
We know that the essential image of $\Coh(\cY) \to \Coh(\hat{\cY})$ is stable under kernels, cokernels and extensions.
Since $\hat{g}_* \hat{g}^* \fG$ is in the essential image  and both the kernel and cokernel of $\fG \to \hat{g}_* \hat{g}^* \fG$ are in the essential image by the d\'evissage hypothesis, we conclude that $\fG$ is in the essential image.
\end{proof}

\begin{corollary} \label{cor:descent}
  Let $f \co \cX \to \cY$ be a universally submersive and representable morphism of noetherian algebraic stacks of finite type over a noetherian $I$-adically complete ring $R$.  Let $\cY_0 \subset \cY$ be a closed substack and $\cX_0 \subset \cX$ be its preimage.  Assume that $\Gamma(\cY, \oh_{\cY}) = R$ and that $\Gamma(\cY_0, \oh_{\cY_0}) = R/I$. Suppose that for each $n \ge 1$, $(\cX/\cY)^n$ is cohomologically proper over $\Spec R$ and that there is a closed substack $\cZ_n \subset (\cX/\cY)^n$ such that $|\cZ_n| \subset |(\cX_0/\cY_0)^n)|$ and such that the pair $((\cX/\cY)^n, \cZ_n)$ is coherently complete and satisfies formal functions.  Then $\cY$ is cohomologically proper over $\Spec R$ and the pair $(\cY, \cY_0)$ is coherently complete and satisfies formal functions.

  If $f \co \cX \to \cY$ is surjective, proper, and representable, then in each case one only needs to require the properties for $n=1$. 
\end{corollary} 

\begin{proof}
  The corollary follows directly from \Cref{thm:descent} using \Cref{lem:permanence-under-larger-substacks} to deduce that the pair $((\cX/\cY)^n, (\cX_0/\cY_0)^n)$ is coherently complete and satisfies formal functions.
\end{proof}

\section{$\bG_m$-actions and destabilizing one-parameter subgroups}

After reviewing properties of $\bG_m$-actions and one parameter subgroups, we establish a refinement of the stabilization theorem in GIT (sometimes called the Hilbert--Mumford criterion) establishing the existence of destabilizing one-parameter subgroups that are regular in the stabilizer of the limit (\Cref{prop:refined-destabilization}).

\subsection{The fixed and attractor
  subschemes} \label{sec:fixed-attractor} Let $G$ be a connected,
smooth, and affine group scheme of finite type over a noetherian ring
$R$.  Let $X = \Spec A$ be an affine scheme over $R$ with an action of
$G$.  For any one-parameter subgroup $\lambda \co \bG_m \to G$, we
define the subfunctors
\begin{eqnarray*} 
  X^0_\lambda 
    &\coloneqq
    & \underline{\Hom}_R^{\bG_m}(\spec R, X) \hspace{0.45cm} \text{(fixed locus)} \\
  X^+_\lambda 
    &\coloneqq
    & \underline{\Hom}_R^{\bG_m}(\bA^1, X) \hspace{1.05cm} \text{(attractor)} 
\end{eqnarray*}
of $X$ as introduced in \cite{drinfeld}. These functors are represented by closed subschemes of $X$. Indeed, the $\lambda$-action on $X$ induces a $\bZ$-grading $A = \bigoplus_n A_n$.  Let
$I^-_\lambda = \sum_{n < 0} A_n$ and $I^+_{\lambda} = \sum_{n > 0} A_n$ denote the ideals generated by homogeneous elements in strictly negative and strictly positive degree; then one can check that $X^0_\lambda$ and $X^+_\lambda$ are represented by the closed subschemes $\Spec A/(I^-_\lambda+I^+_\lambda)$ and  $\Spec A/I^-_{\lambda}$, respectively.
If $k$ is an algebraically closed field over $R$, then the $k$-points of $X^0_{\lambda}$ are the $\bG_m$-fixed points  the $k$-points of $X^+_\lambda$ can be described as
\[
  X^+_\lambda(k) = 
  \{x \in X(k) \, \, \vert \, \, \lim_{t \to 0}  \lambda(t) \cdot x \, \, \text{exists} \}.
\]
The inclusion $X^0_{\lambda} \into X^+_{\lambda}$ has a natural retraction given by 
$$\ev_0 \co X^+_{\lambda} \to X^0_{\lambda}, \qquad x \mapsto \lim_{t \to 0} \lambda(t) \cdot x.$$

\begin{remark} \label{R:reduced} 
  If $X = \Spec A$ is integral, then $X^0_{\lambda}$ and $X^+_{\lambda}$ may be non-reduced and reducible.
  For example, if $\bG_m$ acts on $\Spec k[x,y,z,w]/(xw-y^2z^2)$ with weights $(1,0,0,-1)$, then $X^0_{\lambda} = \Spec k[y,z]/(y^2z^2)$ and $X^+_{\lambda} = \Spec k[x,y,z]/(y^2z^2)$.  If $X$ is integral, it can also happen that $X^0_{\lambda}(k)$ is a finite set with more than one point; e.g., if $X = \Spec k[x,y,z]/(x(x-1)+yz)$ with weights $(0,1,-1)$, then $X^0_{\lambda} = \Spec k[x]/(x(x-1))$.
\end{remark}

\subsection{Centralizer, parabolic and unipotent subgroups} 
\label{subsec:centralizer-parabolic}
In this subsection, we recall the dynamic approach to algebraic groups
with respect to one-parameter subgroup as discussed in \cite[Chapter
2.1]{ref:cgp} (also see \cite[\S\S 4-5]{ref:conrad_red}).  Let $G$ be a
connected, smooth, and affine group scheme of finite type over a noetherian ring
$R$.  Let $G$ act on itself via conjugation. Then a one-parameter
subgroup $\lambda \co \bG_m \to G$ induces an action of $\bG_m$ on $G$
(again via conjugation). Consider the following subgroups of $G$:
\begin{align*}
  Z_{\lambda}
  &=\{g \in G \, \vert  \, \, \text{$\lambda(t)g\lambda(t)^{-1} = g$ for all $t \in \bG_m$}\}
  &&\text{(centralizer of $\lambda$)}, \\
  P_{\lambda}
  &= \{g \in G \, \vert \, \lim_{t \to 0} \lambda(t) g \lambda(t)^{-1} \, \, \text{exists} \}
  &&\text{(parabolic of $\lambda$), } \\
  U_{\lambda}
  &= \{g \in G \, \vert \, \lim_{t \to 0} \lambda(t) g \lambda(t)^{-1} = 1 \}
  &&\text{(unipotent of $\lambda$)}, \
\end{align*}
which extend naturally to subgroup functors of $G$. When it is necessary to record the group $G$, we write $Z^{G}_{\lambda}$, $P^{G}_{\lambda}$, and $U^{G}_{\lambda}$.  
These functors are represented by closed subgroup schemes of $G$.
Indeed, observe that with the notation introduced in \S \ref{sec:fixed-attractor}, we have that $Z_{\lambda} = G^0_{\lambda}$ and $P_{\lambda} = G^+_{\lambda}$, and that $U_{\lambda}$ is identified with the kernel of $P_{\lambda}= G^+_{\lambda} \to G^0_{\lambda} = Z_{\lambda}$ defined by $g \mapsto \lim_{t \to 0} \lambda(t) g \lambda(t)^{-1} $. See also \cite[Lems. 2.1.4 and 2.1.5]{ref:cgp}. 

There is a split exact sequence
\begin{equation} \label{E:UPZ} 
  1 \to U_{\lambda}\to P_{\lambda}\to Z_{\lambda}\to 1.
\end{equation}
Over a field, the groups $Z_{\lambda}$, $P_{\lambda}$ and
$U_{\lambda}$ are well-known to satisfy several nice properties. If
$G$ is smooth (resp. connected), so is $Z_{\lambda}, P_{\lambda}$ and
$U_{\lambda}$.  The group $U_{\lambda}$ is unipotent.  If $G$ is
connected and reductive then so is $Z_{\lambda}$
\cite[Thm.~C.2.1]{ref:conrad_alg_notes}. Moreover, in this case,
$U_{\lambda}$ is the unipotent radical of $P_{\lambda}$ and
$G/P_{\lambda}$ is projective \cite[Prop.~2.2.9]{ref:cgp}. In general,
if $G$ is a split reductive group scheme over a noetherian ring $R$
and $\lambda \co \bG_{m} \to G$ is a one-parameter subgroup, then
$Z_{\lambda}$, $P_{\lambda}$, and $U_{\lambda}$ are closed subgroup
schemes of $G$ smooth over $R$, $G / P_{\lambda}$ is projective over
$R$ and $Z_\lambda$ is reductive \cite[Thm.~4.1.7, Ex.~4.1.9,
Cor.~5.2.8]{ref:conrad_red}.

If $G$ acts on an affine scheme $X$ of finite type over $R$, there are
natural actions of $Z_{\lambda}$ and $P_{\lambda}$ on $X^0_\lambda$
and $X^+_\lambda$ respectively.  The evaluation map
$\ev_0 \co X^+_{\lambda} \to X^0_{\lambda}$ is equivariant with
respect to $P_{\lambda}\to Z_{\lambda}$, and induces a morphism
$\ev_0 \co [X^+_{\lambda}/P_{\lambda}] \to[X^0_{\lambda}/Z_{\lambda}]$
on quotient stacks.

\subsection{Positively graded
  actions} \label{subsec:postively-graded-sections} Here we consider a
distinguished class of actions that appear frequently in practice (see
\Cref{E:plus-is-semipositive}).
\begin{definition} 
  Recall that an action of $\bG_m$ on an affine scheme $\Spec A$ over
  a noetherian ring $R$ is given by a grading $A = \bigoplus_{d}
  A_d$. The action is \emph{positively graded}
  (resp. \emph{semipositively graded}) if $A_d = 0$ for $d < 0$ and
  $A_0$ is finite over $R$ (resp., if $A_d = 0$ for $d < 0$).  A
  representation $V = \bigoplus_d V_d$ of $\bG_m$ is \emph{positively
    graded} (resp., \emph{semipositively graded}) if $V_d = 0$ for
  $d \le 0$ (resp., if $V_d = 0$ for $d < 0$).

  For the action of an algebraic group $G$ on an affine scheme
  $\Spec A$ over $R$ and a one-parameter subgroup
  $\lambda \co \bG_m \to G$, we say that $\Spec A$ is
  \emph{$\lambda$-positive} (resp. \emph{$\lambda$-semipositive}) if
  the $\bG_m$-action on $\Spec A$ induced by $\lambda$ is positively
  graded (resp. semipositively graded).
\end{definition}

\begin{lemma} \label{lem:positively-graded-equivalence}
  Let $\bG_m$ act on an affine scheme $X=\Spec A$ of finite type
  over a field $k$.
  \begin{enumerate}[{\rm(1)}]
  \item The action is semipositively graded if and only if
    $X^+_{\id} = X$.
  \item The action is positively graded if and only if $X^{\bG_m}$ is finite over $\Spec k$.
  \end{enumerate}
\end{lemma}

\begin{proof}
  The first part is clear from the definitions.  For the second, observe that the condition that $A_d = 0$ for $d < 0$ implies that $X^{\bG_m} = \Spec (A/I^+) = \Spec A_0$. \end{proof}
  
The following is our key example.
\begin{example}\label{E:plus-is-semipositive}
  If $G$ is an algebraic group acting on an affine scheme $X = \Spec A$ of finite type over $k$ and $\lambda \co \bG_m \to G$ is a one-parameter subgroup, then $X^+_{\lambda}$ is always $\lambda$-semipositive.  Moreover, $X^+_{\lambda}$ is $\lambda$-positive if and only if $X^0_{\lambda}$ is finite over $\Spec k$.
\end{example}

\begin{lemma} \label{lem:positivity-under-products} Let $G$ be an
  affine group scheme of finite type over a noetherian ring $R$.  Let
  $\Spec A$ and $\Spec B$ be affine schemes of finite type over $R$
  with actions of $G$.  Let $\lambda \co \bG_{m} \to G$ be a
  one-parameter subgroup.  If $A$ is $\lambda$-positive with
  $A^{\lambda}=R$ and $B$ is $\lambda$-positive
  (resp. $\lambda$-semipositive), then $A \tensor_R B$ is
  $\lambda$-positive (resp. $\lambda$-semipositive) and
  $(A \tensor_R B)^G = B^G$.
\end{lemma}

\begin{proof}
  Let $A = \bigoplus_{d \ge 0} A_d$ and $B = \bigoplus_{d \ge 0} B_d$ be the gradings induced by the $\lambda$-action.  Then $(A \tensor_R B)_d = \bigoplus_{i + j = d} A_i \tensor B_j$ is $\lambda$-semipositive with 
  $$(A \tensor_R B)^{\lambda} = A^{\lambda} \tensor_R B^{\lambda} = B^{\lambda}.$$
  It follows that $(A \tensor_R B)^G = B^G$.
\end{proof}

\begin{example} \label{ex:lambda-unipotents-are-positively-graded} Let
  $G$ be a reductive algebraic group over an algebraically closed
  field $k$. If $\lambda \co \bG_m \to G$ is a one-parameter subgroup,
  then the $\lambda$-grading on
  $\Gamma(U_{\lambda})=\Gamma(U_{\lambda}, \oh_{U_{\lambda}})$ is
  positively graded and $\Gamma(U_{\lambda})^{\lambda} = k$.  To see
  this, choose a maximal torus $T$ containing $\lambda$. Since the
  root system $\Phi(G,T)$ is reduced \cite[Cor.\
  2.2.1]{ref:conrad_alg_notes}, there is a decomposition as schemes
  $U_{\lambda}\simeq \prod_a U_a$ \cite[Cor.\ 3.3.12]{ref:cgp}, where
  the product is over all root groups $U_a$ with $a \in \Phi(G,T)$
  satisfying $\langle a, \lambda \rangle > 0$. Each $U_a \simeq \bG_a$
  as group schemes over $k$, hence
  $\Gamma(U_{\lambda}, \oh_{U_{\lambda}})$ is a polynomial ring in the
  variables $x_a$ indexed by $a \in \Phi(G,T)$, where the
  $\lambda$-weight of $x_a$ is $\langle a, \lambda \rangle$
  \cite[Prop.\ 2.1.8]{ref:cgp}. Therefore, the $\lambda$-weight of a
  monomial $x_{a_1}^{m_1}\cdots x_{a_k}^{m_k}$ is strictly positive
  unless $m_1 = \cdots = m_k = 0$.  It then follows from
  \Cref{lem:positivity-under-products} that the $\lambda$-grading on
  $\Gamma(U_{\lambda})^{\tensor n}$ is also positively graded for
  $n \ge 1$ and that
  $(\Gamma(U_{\lambda})^{\tensor n})^{\lambda} = k$.

  More generally, if $G$ is a split reductive group scheme over a
  noetherian ring $R$ and $\lambda$ is a one-parameter subgroup
  contained in a maximal torus $T \subset G$, then
  $\Gamma(U_{\lambda})$ is $\lambda$-positive with
  $\Gamma(U_{\lambda})^{\lambda} =
  R$.
\end{example}

\subsection{Regular one-parameter subgroups}
If $T$ is a torus over a noetherian ring $R$, there is a
perfect bilinear pairing between the \emph{character lattice}
$\bX^\ast(T) \coloneqq \Hom(T,\bG_m)$ and the \textit{lattice of
  one-parameter subgroups} (or \emph{cocharacter lattice})
$\bX_\ast(T) \coloneqq \Hom(\bG_m, T)$
\begin{equation}\label{eq:pairing} 
  \langle \cdot, \cdot \rangle \co 
  \bX^\ast(T) \times \bX_\ast(T) 
  \to \operatorname{End}(\bG_m) \simeq \bZ
\end{equation}
given by $\langle a, \lambda \rangle = a \circ \lambda$.

Recall that if $T$ is a maximal torus of $G$, the \textit{root system}
of the pair $(G,T)$ is the subset
$\Phi(G,T) \subset \bX^\ast(T) - \{0\}$ consisting of non-trivial
weights for the adjoint action of $T$ on
$\operatorname{Lie}(G)$. 

\begin{definition}
  A one-parameter subgroup $\lambda \co \bG_m \to T$ is
  \textit{regular with respect to $T$} if
  $\langle a, \lambda \rangle \neq 0$ for all $a \in \Phi(G,T)$.
\end{definition}

\begin{remark}
  When $G = \GL_n$ and $T$ is the diagonal torus in $G$, then the Lie algebra ${\rm Lie}(G) = {\rm Mat}_{n,n}$ is the vector space of $n \times n$ matrices. The basis element $E_{i,j}$ with a 1 in position $(i,j)$ and 0 elsewhere has weight $\chi_i - \chi_j$, where $\chi_i \co T \to \bG_m$ denotes the character defined by $(t_k) \mapsto t_i$.  Therefore a one-parameter subgroup $\lambda=(\lambda_k) \co \bG_m \to T$ is regular if and only if the $\lambda_i$'s are distinct.
\end{remark}

\begin{lemma}\label{L:regular} 
  Let $G$ be a connected, smooth, and affine group scheme  over a noetherian ring $R$. Let $T \subset G$ be a
  maximal torus and $\lambda$ a one-parameter subgroup of $T$.
  \begin{enumerate}[{\rm(1)}]
    \item \label{L:regular_1} 
     $\lambda$ is regular with respect to $T$ if and only if $\Cent^G(T) = Z_{\lambda}$. 
   \item \label{L:regular_2} If $G$ is reductive, then $\lambda$ is regular with respect to $T$ if and only if
     $Z_{\lambda}$ is a maximal torus in $G$.
   \item \label{L:regular_3} If $T' \subset G$ is another maximal
     torus containing the image of $\lambda$, then $\lambda$ is
     regular with respect to $T$ if and only the same is true with
     respect to $T'$. In other words, the definition of regularity is
     independent of the maximal torus $T$ chosen. 
  \end{enumerate}
\end{lemma}

\begin{proof}
  Since $G$ is smooth connected, so too are $\Cent^G(T)$
  \cite[Lem.~2.2.4]{ref:conrad_red} and $Z_{\lambda}$. Therefore, the
  containment $\Cent^G(T) \subset Z_{\lambda}$ is an equality if and
  only if it is true on geometric fibers. In particular, we may assume
  that $R$ is an algebraically closed field. In this case, it suffices
  to prove that
  $\operatorname{Lie}(\Cent^G(T)) =
  \operatorname{Lie}(Z_{\lambda})$. Now
  \[
\operatorname{Lie}(G)_0=\operatorname{Lie}(G)^T=\operatorname{Lie}(\Cent^G(T)) \subset  \operatorname{Lie}(Z_{\lambda}) = \operatorname{Lie}(G)^\lambda =\oplus_{a\in \Phi(G,T)\cup \{0\}\,:\,\langle
    a, \lambda \rangle = 0} \operatorname{Lie}(G)_a.
  \]
  This proves that $\lambda$ is regular if and only if
  $\Cent^G(T) = Z_{\lambda}$. Moreover if $G$ is reductive then
  $\Cent^G(T) = T$ \cite[Prop.\ 2.3.1]{ref:conrad_alg_notes}, proving
  \eqref{L:regular_1} and \eqref{L:regular_2}.

  To prove \eqref{L:regular_3}, suppose that the image of $\lambda$ is
  contained in some other maximal torus $T'$. By the preceding
  discussion, it suffices to show that
  $\dim \Cent^G(T') = \dim Z_{\lambda}$ when $R$ is an algebraically
  closed field. We have $\dim Z_{\lambda}= \dim \Cent^G(T)$ since
  $\lambda$ is regular with respect to $T$. On the other hand, all
  maximal tori of $G$ are $R$-conjugate, so
  $\dim \Cent^G(T') = \dim \Cent^G(T)$ and hence $\lambda$ is also
  regular with respect to $T'$.
\end{proof}

\subsection{Destabilizing one-parameter subgroups}

A degeneration from a non-closed orbit to a closed orbit can be realized by a one-parameter subgroup.  This is the classical destabilization theorem in GIT, which implies the Hilbert--Mumford criterion \cite[p.~53]{git}.

\begin{proposition} \label{P:destabilization}
  Let $G$ be a reductive algebraic group over an algebraically closed field $k$.  Let $X$ be an affine scheme of finite type over $k$ with an action of $G$.  For any point $x \in X(k)$, there is a one-parameter subgroup $\lambda \co \bG_m \to G$ such that $\lim_{t \to 0} \lambda(t) \cdot x$ exists and has closed orbit.
\end{proposition}

When $X$ is integral and has a unique closed orbit, it is possible to find a one-parameter subgroup destabilizing every orbit.

\begin{proposition} \label{P:generic-destabilization}
  Let $G$ be a reductive algebraic group over an algebraically closed field $k$.  Let $X = \Spec A$ be an integral affine scheme of finite type over $k$ with an action of $G$. Assume that there is a unique closed orbit $G$-orbit $Gx_0 \subset X$.  
  Then there exists a one-parameter subgroup $\lambda \co \bG_m \to G$ such that for any point $x \in X(k)$, there exists $g \in G(k)$ such that $\lim_{t \to 0} \lambda(t) \cdot gx \in Gx_0$.  In particular, $G \cdot X^+_\lambda = X$ and $\emptyset \neq X^0_\lambda \subset Gx_0$. 
\end{proposition}

\begin{proof}
  Let $K = K(X)$ be the function field of $X$ and let $\xi \in X(K)$
  be the generic point.  Letting $K \to \overline{K}$ be an algebraic
  closure, we consider the base change $X \otimes_k \overline{K}$ with
  the point
  $\overline{\xi} \in (X \otimes_k \overline{K})(\overline{K})$
  induced from $\xi$. By \Cref{P:destabilization}, there is a
  one-parameter subgroup $\overline{\lambda} \colon \bG_{m, \overline{K}} \to G_{\overline{K}}$
  such that $\lim_{t \to 0} \overline{\lambda}(t) \cdot \overline{\xi}$ exists and has
  closed $G_{\overline{K}}$-orbit.  If $T \subset G$ is a maximal torus,
  then there exists an element $g \in G(\overline{K})$ such that
  $g \overline{\lambda} g^{-1} \in T_{\overline{K}}$.  Since extension of scalars gives an
  isomorphism
  \[
    \bZ^n \cong \Hom_{k}(\bG_{m}, T) 
    \iso \Hom_{\overline{K}}(\bG_{m, \overline{K}}, T_{\overline{K}}) \cong \bZ^n,
  \]
  there is a one-parameter subgroup $\lambda \co \bG_m \to T$ with
  $\lambda_{\overline{K}} = g \overline{\lambda} g^{-1}$.  Since the
  generic point $\xi \in [X/G](K)$ is in the image of the proper map
  $[X^+_{\lambda}/P_{\lambda}] \to [X/G]$, this map is surjective.
  Moreover, every point $x \in [X/G](k)$ has a representative in
  $X(k)$ which specializes via $\lambda$ to the unique closed orbit.

  Alternatively, we can appeal to the Kempf--Hesselink stratification.  Kempf proved in \cite[Thm.~3.4]{ref:kempf} that for any point  $x \in X(k)$, there is a  destabilizing one-parameter subgroup $\lambda_x \co \bG_m \to G$ (unique up to conjugation by the unipotent radical $U_{\lambda_x}$) with $\lim_{t \to 0} \lambda_x(t) x \in G x_0$ minimizing the normalized Hilbert--Mumford index $\mu(x,-) / ||-||$, where $|| - ||$ is a fixed conjugation-invariant norm on $\bX_*(G)$. 
  In \cite{hesselink}, Hesselink used Kempf's optimal destabilizing one-parameter subgroup 
  to show that $X \setminus G x_0 = \coprod S_{\lambda_i,M_i}$ admits a stratification into locally closed $G$-invariant subschemes where each $\lambda_i$ is a one-parameter subgroup and $M_i < 0$.  For any $k$-point $x \in S_{\lambda_i,M_i}$ there is an element $g \in G(k)$ such that $\lambda_i$ is an optimal one-parameter subgroup for $gx$ with $\mu(gx,\lambda_i)/||\lambda_i|| = M_i$.  Since $X$ is integral, there is a generic strata $S_{\lambda, M}$ for some $(\lambda, M) \in \{(\lambda_i, M_i) \}$, which is dense.  Since $[X^+_{\lambda}/P_{\lambda}] \to [X/G]$ is proper and its image contains the dense set $[S_{\lambda,M}/G]$, it must be surjective.  The statement follows.
\end{proof}

\begin{example}
  Under the $\SL_n$ action on $\bA^n$ given by the standard representation, the origin is the unique closed orbit and the complement $\bA^n \setminus 0$ is a dense orbit.  Kempf's  optimal one-parameter subgroup for $(1, 0, \ldots, 0)$ is $\lambda(t) = \mathrm{diag}(t^{n-1}, t^{-1}, \ldots, t^{-1})$. Observe that this subgroup is unfortunately not regular, which is a property we desire for our application.  However, nearby deformations are regular: for  distinct positive integers $d_1, \ldots, d_{n-1}$ with sum $d$, the one-parameter subgroup $\lambda'(t) =  \mathrm{diag}(t^{d}, t^{-d_1}, \ldots, t^{-d_{n-1}})$ is regular and also destabilizes the generic orbit.
\end{example}

In fact, one can always find destabilizing one-parameter subgroups $\lambda$ of a point $x$ such that $\lambda$ is regular in the stabilizer of the limit $x_0 = \lim_{t \to 0} \lambda(t) \cdot x$.  This is proven in \Cref{prop:refined-destabilization}, which relies on the following lemma.

\begin{lemma} \label{lem:deformation-1ps}
  Let $G$ be a reductive algebraic group over an algebraically closed field $k$.  Let $X$ be an affine scheme of finite type over $k$ with an action of $G$ and let $x \in X(k)$. Let $\lambda \co \bG_m \to G$ be a one-parameter subgroup such that $x_0 := \lim_{t \to 0} \lambda(t) x$ has closed orbit.  If $\rho \co \bG_m \to G_{x_0}$ is a one-parameter subgroup commuting with $\lambda$ (e.g., $\rho$ and $\lambda$ are contained in the same maximal torus of $G_{x_0}$), then for $n \gg 0$
  $$\lim_{t \to 0} (\lambda^n \rho)(t) \cdot x \in G x_0.$$
\end{lemma}

\begin{proof}
  We first determine the pushout of the diagram:
  \begin{equation} \label{eqn:pushout}
    \begin{split}
      \xymatrix{
        B \bG_m \ar@{^(->}[r] \ar@{^(->}[d]
          & [\bA^1/\bG_m] \\
        [\bA^1/\bG_m] \times B \bG_m,
      }
    \end{split}
  \end{equation}
  where the top map $B\bG_m \into [\bA^1/\bG_m]$ is the closed immersion defining the origin and the left map $B \bG_m \into [\bA^1/\bG_m] \times B \bG_m$ is the open immersion corresponding to the product of $B \bG_m$ and the inclusion of the open point $1 \co \Spec k \into [\bA^1/\bG_m]$. Expressing the diagram $\bG_m^2$-equivariantly, we can write $B \bG_m = [\Spec k[y]_y / \bG_m^2]$, $[\bA^1/\bG_m] \times B \bG_m = [\Spec k[y]/\bG_m^2]$, and $[\bA^1/\bG_m] = [\Spec k[y,z]_y / \bG_m^2]$ under the diagonal action $(t_1,t_2) \cdot (y,z) = (t_1 y, t_2z)$.  The $\bG_m^2$-equivariant pushout is determined by the fiber product of rings
  $$\xymatrix{
    k[y]_y
      & k[y,z]_y \ar[l]_{0 \mapsfrom z} \\
    k[y] \ar[u] 
      & k[y, \frac{z}{y}, \frac{z}{y^2}, \ldots]. \ar[l] \ar[u]
  }$$
  We can write $k[y, \frac{z}{y}, \frac{z}{y^2}, \ldots] = \bigcup_n k[y,z_n]$ with $z_n =\frac{z}{y^n}$, where $\bG_m^2$ acts via $(t_1, t_2) \cdot (y,z_n) = (t_1 y, \frac{t_2}{t_1^n} z_n)$.  Identifying $\bA^2 = \Spec k[y,z_n]$, then $\bG_m^2$ acts via the degree matrix $D_n = \begin{pmatrix} 1 & 0 \\ -n & 1 \end{pmatrix}$.  By generalities of pushouts of stacks (see \cite[\S 4]{ref:ahhlr}), the pushout of \eqref{eqn:pushout} is a limit
  $$\limit_n \, [\bA^2/_{D_n} \bG_m^2]$$
  of quotient stacks with affine transition morphisms.

  The point $x \in X(k)$ and one-parameter subgroup $\lambda$ define a morphism $\lambda_x \co [\bA^1/\bG_m] \to [X/G]$ with $\lambda_x(1) \simeq x$ and $\lambda_x(0) \simeq x_0$. Since $\rho$ commutes with $\lambda$, we also have a morphism $(\lambda_x, \rho) \co [\bA^1/\bG_m] \times B \bG_m \to [X/G]$ giving a commutative diagram of solid arrows
  $$      \xymatrix{
    B \bG_m \ar@{^(->}[r]^{z=0} \ar@{^(->}[d]^{y \neq 0}
      & [\bA^1/\bG_m] \ar[d] \ar@/^1pc/[ddr]^{\lambda_x} \\
    [\bA^1/\bG_m] \times B \bG_m \ar[r] \ar@/_1pc/[rrd]_{(\lambda_x,\rho)}
      & [\bA^2/_{D_n} \bG_m^2] \ar@{-->}[dr]^{\Psi} \\
    & & [X/G].
  }$$
  Since $[X/G]$ is of finite type, the induced morphism $\limit_n  [\bA^2/_{D_n} \bG_m^2] \to [X/G]$ factors through $[\bA^2/_{D_n} \bG_m^2]$ for $n \gg 0$, and thus defines the dotted arrow $\Psi \co [\bA^2/_{D_n} \bG_m^2]  \to [X/G]$.

  Now consider the map $[\bA^1/\bG_m] \to [\bA^2/_{D_n} \bG_m^2]$ induced by the diagonal $\bA^1 \to \bA^2$ and the group homomorphism $\bG_m \to \bG_m^2$ defined by $t \mapsto (t, t^{n+1})$.  The composition $[\bA^1/\bG_m] \to [\bA^2/_{D_n} \bG_m^2] \xto{\Psi} [X/G]$ defines a morphism such that 0 maps to $x_0$ and such that the induced map $\bG_m \to G_{x_0}$ on stabilizers is given by $\lambda^{n+1} \rho$.  The statement follows.
\end{proof}

\begin{proposition} \label{prop:refined-destabilization}
  Let $G$ be a reductive algebraic group over an algebraically closed field $k$, and let $X$ be an affine scheme of finite type over $k$ with an action of $G$.  For any point $x \in X(k)$, there is a one-parameter subgroup $\lambda \co \bG_m \to G$ such that $x_0 := \lim_{t \to 0} \lambda(t) \cdot x$ has closed orbit and such that the induced one-parameter subgroup $\lambda \co \bG_m \to G_{x_0}$ is regular.  

  Moreover, if there is a unique closed orbit $Gx_0$ and $X$ is integral, then there is a one-parameter subgroup $\lambda \co \bG_m \to G$ such that for any point $y \in X(k)$, there exists $g \in G(k)$ such that $y' := \lim_{t \to 0} \lambda(t) \cdot gy \in Gx_0$ and such that the induced map $\lambda \co \bG_m \to G_{y'}$ is a regular one-parameter subgroup.
\end{proposition}

\begin{proof}
  Let $\lambda \co \bG_m \to G$ be a one-parameter subgroup such that $x_0 := \lim_{t \to 0} \lambda(t) \cdot x$ has closed orbit (\Cref{P:destabilization}).  Let $T \subset G_{x_0}$ be a maximal torus containing the image of $\lambda$.  Choose a one-parameter subgroup $\rho \co \bG_m \to T$ such that the composition $\lambda^n \rho \co \bG_m \to T \to G_{x_0}$ is a regular one-parameter subgroup for $n \gg 0$. \Cref{lem:deformation-1ps} implies that $x_0 = \lim_{t \to 0} (\lambda^n \rho)(t) \cdot x$ for $n \gg 0$.
  
  The addendum follows from applying the main statement to the generic point $\xi \in X$ as in the proof of \Cref{P:generic-destabilization}.   There exist a one-parameter subgroup $\lambda \co \bG_m \to G$, a finite extension $K=K(X) \to K'$, and $g \in G(K')$ such that $g \xi' \in X^+_{\lambda}(K')$ and $\xi'_0 := \lim_{t \to 0} \lambda(t) \cdot g \xi'$ has closed orbit in $X_{K'}$, where $\xi' \in X(K')$ is the lift of $\xi \in X(K)$. Choose a maximal torus $T \subset G$ containing $\lambda$ and a one-parameter subgroup $\rho \co \bG_m \to T$ such that the base change of $\lambda^n \rho$ to $K'$ is regular in $G_{\xi'_0}$ for $n \gg 0$.  \Cref{lem:deformation-1ps} implies that $\lambda^n \rho$ also destabilizes $g \xi'$.   It follows that any $G$-orbit of a $k$-point in $X$ contains a $k$-point $y$ such that $y' := \lim_{t \to 0} (\lambda^n \rho)(t) \cdot y \in Gx_0$.  Since the base change of $G_{y'} = \Stab^G(y')$ to $K'$ is conjugate to $(G_{K'})_{\xi'_0} = \Stab^{G_{K'}}(\xi'_0)$, the one-parameter subgroup $\lambda^n \rho \co  \bG_m \to G_{y'}$ is also regular. 
\end{proof}

\begin{remark}
  One can also formulate this result as saying that the degeneration fan \linebreak
  $\mathcal{D}eg([\Spec A/G], x)$, introduced in \cite{hl-instability}, has dimension at least the reductive rank of $G_{x_0}$.
\end{remark}

\section{Proofs of coherent completeness}

In this section, we prove the main coherent completeness theorem (\Cref{thm:main}) as well as the coherent completeness result for positively graded actions (\Cref{thm:positive-graded}).

\subsection{Proof of the main theorem}
\begin{proof}[Proof of \Cref{thm:main}\eqref{thm:main1}]
  Let $\cX = [X/G]$.  By \Cref{thm:descent}, we can replace $\cX$ with $\cX_{\red}$ and $R$ with $R_{\red}$, and we can thus assume that $R = k$ is a field and $\cX$ is reduced.  
  If $G^0$ denotes the connected component of the identity, then $[X/G^0] \to [X/G]$ is finite and \'etale.  Applying \Cref{thm:descent} reduces us to the case that $G$ is connected.  Letting $[\Spec A/G] = \coprod_i \cX_i$ be the irreducible decomposition,  then each $\cX_i$ is isomorphic to $[\Spec A_i / G]$ with $A_i^G = k$.  The morphism $\coprod_i \cX_i \to [X/G]$ is a proper and surjective, and by applying \Cref{thm:descent} again, we are further reduced to the case that $A$ is an integral domain. 

  Since $[\Spec A/G]$ has a unique closed point, \Cref{prop:refined-destabilization} and limit methods 
  imply that there is a one-parameter subgroup $\lambda \co \bG_{m,k'} \to G_{k'}$ where $k \to k'$ is a finite field extension such that $G \cdot (X_{k'})^+_{\lambda} = X_{k'}$ and such that for any closed point $y \in (X_{k'})^+_{\lambda}$, the limit $y' := \lim_{t \to 0} \lambda(t) \cdot gy$ has closed orbit and the induced map $\lambda \co \bG_{m,k'} \to G_{y'}$ is a regular one-parameter subgroup.  By \Cref{thm:descent}, we may replace $k$ with $k'$ and thus we can assume that $\lambda$ is defined over $k$.  We can further assume that the unique closed point $x \in [X/G]$ lifts to a $k$-point of $X$.
  
  We consider the composition
  $$[X^+_{\lambda}/Z_{\lambda}] \xto{f} [X^+_{\lambda}/P_{\lambda}] \xto{g} [X/G].$$
  The morphism $g$ is proper (since $X^+_{\lambda} \subset X$ is a closed subscheme and $G/P_{\lambda}$ is projective) and surjective (since $G \cdot X^+_\lambda = X$).  By applying \Cref{thm:descent}, we are reduced to showing that $([X^+_{\lambda}/P_{\lambda}], g^{-1}(\cG_x))$ satisfies formal functions and is coherent complete.  Since $g$ is proper, the pushforward $g_* \oh_{[X^+_{\lambda}/P_{\lambda}]}$ is coherent and it follows that $\Gamma(X^+_{\lambda})^{P_{\lambda}}$ is finite over $k$ (where $\Gamma(X^+_{\lambda})$ is shorthand for $\Gamma(X^+_{\lambda}, \oh_{X^+_{\lambda}})$).

  On the other hand, the morphism $f$ is faithfully flat and of finite type. The higher base changes of $f$ can be computed as 
  \begin{equation} \label{eqn:higher-base-changes}
    ([X^+_{\lambda}/Z_{\lambda}] \, / \, [X^+_{\lambda}/P_{\lambda}])^n
    \cong [(X^+_{\lambda} \times (U_{\lambda})^{n-1}) / Z_{\lambda}].
  \end{equation}
  Since the composition $[X^0/Z_{\lambda}] \to [X^+/P_{\lambda}] \xto{\ev_0} [X^0/Z_{\lambda}]$ is the identity, the map $\Gamma(X^0_{\lambda})^{Z_{\lambda}} \to \Gamma(X^+_{\lambda})^{P_{\lambda}}$ is injective and thus $\Gamma(X^0_{\lambda})^{Z_{\lambda}}$ is also finite over $k$.  But since $\Gamma(X^+_{\lambda})^{\lambda} = \Gamma(X^0_{\lambda})$, we can conclude that $\Gamma(X^+_{\lambda})^{Z_{\lambda}} = \Gamma(X^0_{\lambda})^{Z_{\lambda}}$ is finite over $k$.  Since $Z_{\lambda}$ is reductive, the map $[X^+_{\lambda}/Z_{\lambda}] \to \Spec \Gamma(X^+_{\lambda})^{Z_{\lambda}}$ is an adequate moduli space.  The stack $[X^+_{\lambda}/Z_{\lambda}]$ has finitely many closed points $y_1, \ldots, y_n \in X_{\lambda}^0$ with each $y_i$ contained in $(g \circ f)^{-1}(\cG_x)$. 
  By construction, each one-parameter subgroup $\lambda \co \bG_m \to \Stab^G(y_i)$ is regular.  It follows that the stabilizer $\Stab^{Z_{\lambda}}(y_i) = Z_{\lambda} \cap \Stab^G(y_i)$ is a maximal torus of $\Stab^G(y_i)$, and in particular linearly reductive.  By \cite[Thm.~4.21]{ref:ahr2}, $[X^+_{\lambda}/Z_{\lambda}] \to \Spec \Gamma(X^+_{\lambda})^{Z_{\lambda}}$ is a good moduli space.  Thus $[X^+_{\lambda}/Z_{\lambda}]$ is coherently complete and satisfies formal functions with respect to $\cG_{y_1} \cup \cdots \cup \cG_{y_n}$, and therefore also coherently complete and satisfies formal functions with respect to $(g \circ f)^{-1}(\cG_x)$ (\Cref{lem:permanence-under-larger-substacks}).

  Similarly, we claim that the higher base changes \eqref{eqn:higher-base-changes} are coherently complete and satisfied formal functions with respect to the preimage of $\cG_x$. Since $\Gamma(U_{\lambda})^{\lambda} = k$ (\Cref{ex:lambda-unipotents-are-positively-graded}), the global sections of $([X^+_{\lambda}/Z_{\lambda}] \, / \, [X^+_{\lambda}/P_{\lambda}])^n$ are identified with the global sections of $[X^+_{\lambda}/Z_{\lambda}]$ (\Cref{lem:positivity-under-products}).  Thus, $([X^+_{\lambda}/Z_{\lambda}] \, / \, [X^+_{\lambda}/P_{\lambda}])^n \to \Spec \Gamma(X^+_{\lambda})^{Z_{\lambda}}$ is an adequate moduli space.  Since the closed points have linearly reductive stabilizers, it is in fact a good moduli space and the claim follows.  
  
  By \Cref{thm:descent}, $[X^+_{\lambda}/P_{\lambda}]$ is coherently complete and satisfies formal functions with respect to the preimage of $\cG_x$, and as we've already observed this implies that $[X/G]$ is coherently complete and satisfies formal functions with respect to $\cG_x$. 
\end{proof}

\begin{remark}
  When $X$ has a $G$-fixed point $x \in X^G$, then there is a more direct argument relying on the generic stabilization theorem (\Cref{P:generic-destabilization}) but not on its refinement (\Cref{prop:refined-destabilization}).  Indeed, as in the proof above, we can reduce to the case that $R=k$ is a field, $G$ is connected, and $X$ is integral.  
  \Cref{P:generic-destabilization} implies that after replacing $k$ with a finite extension, there exists a one-parameter subgroup $\lambda \co \bG_m \to G$ such that $G \cdot X^+_\lambda = X$ and $X^0_{\lambda}=\{x\}$ (set-theoretically).  Let $T \subset G$ be a maximal torus containing $\lambda$ and $B \subset P_{\lambda}$ be a Borel containing $T$. 
  In the composition 
  $$
    [X^+_{\lambda}/T] \xto{f} [X^+_{\lambda}/B] \xto{g} [X/G],
  $$
  the morphism $f$ is proper and surjective, and $g$ is faithfully flat.  The higher base changes of $f$ can be computed as 
  \[ 
    ([X^+_{\lambda}/T] \, / \, [X^+_{\lambda}/B])^n
    \cong [(X^+_{\lambda} \times U^{n-1}) / T]
  \] 
  where $U = B/T$.

  Since $X^+_{\lambda}$ is $\lambda$-positive with $\Gamma(X^+_{\lambda})$ finite over $k$ (\Cref{lem:positively-graded-equivalence}) and $\Gamma(U_{\lambda})^{\lambda} = k$ (\Cref{lem:positively-graded-equivalence}), the global sections of $([X^+_{\lambda}/T] \, / \, [X^+_{\lambda}/B])^n$ are also finite over $k$ (c.f. (\Cref{lem:positivity-under-products})).
  It follows that for each $n \ge 1$, the higher base changes
      $\big([X^+_{\lambda}/T] \, / \, [X^+_{\lambda}/B]\big)^n$
  admit good moduli spaces that are finite schemes over $k$, and that the higher base changes are coherently complete and satisfy formal functions with respect to the preimage of $\cG_x$. The coherent completeness and the formal functions of $([X/G],\cG_x)$ follow from applying \Cref{thm:descent} subsequently to $g$ and $f$.  
\end{remark}
\begin{remark}\label{R:generically-linearly-reductive}
  Standard reductions show that \Cref{conj:main} for all $R$ and $G$
  follows from the case where $R$ is an integral domain and
  $G=\GL_{n,R}$. Applying the method of
  \Cref{thm:main}\eqref{thm:main1} to the fraction field of $R$ and
  replacing $R$ by a finite extension, it follows that it suffices to
  prove \Cref{conj:main} when $[X/G]$ has a saturated and dense open
  with ``nice'' stabilizers \cite[\S2]{ref:ahr2}---in particular, they
  are linearly reductive.
\end{remark}
The following theorem generalizes \Cref{thm:main}\eqref{thm:main2}.

\begin{theorem} \label{thm:reductive}
  Let $G$ be a smooth geometrically reductive group scheme over a noetherian ring $R$.  
  \begin{enumerate}[(1)]
    \item \label{thm:reductive1}
      The morphism $BG \to \Spec R$ is cohomologically proper.
    \item \label{thm:reductive2}
      If in addition $R$ is $I$-adically complete for an ideal $I \subset R$, then the pair $(BG, BG_{R/I})$ satisfies formal functions and is coherently complete.
  \end{enumerate}
\end{theorem}
\begin{proof}
  Since $G \to \Spec R$ is smooth and geometrically reductive, the connected component of the identity $G^0$ is a reductive group scheme over $R$ and $G/G^0$ is finite over $R$.  By applying \Cref{thm:descent} to the finite cover $B G^0 \to B G$, we are reduced to the case that $G \to \Spec R$ is reductive.  After an \'etale cover of $\Spec R$, $G$ becomes split reductive.  Applying the faithfully flat version of \Cref{thm:descent}, we are further reduced to showing that if $G$ is split reductive with maximal torus $T$ and Borel subgroup $B$. Now consider the composition
  \[
    B T \to B B \to B G.
  \]
  Since $G/B$ is projective over $R$, the map $B B \to B G$ is proper.  
  On the other hand, $B T \to BB$ is faithfully flat and the higher base changes can be described as
  $$(BT/BB)^n = [U^{n-1}/T]$$
  where $U = B/T$ is the unipotent radical arising as the quotient of the left action by $T$ and $T$ acts diagonally on $U^{n-1}$ via the right action.  Since $\Gamma(U)^T = R$ (\Cref{ex:lambda-unipotents-are-positively-graded}) and thus $\Gamma(U^{n-1})^T = R$ (\Cref{lem:positivity-under-products}), the map $(BT/BB)^n \to \Spec R$ is a good moduli space for each $n$.  
  Both statements follow by applying \Cref{thm:descent} to $BT \to BB$ and then $BB \to BG$. 
\end{proof}

\subsection{Positively graded group schemes}

Let $G$ be a smooth affine group scheme  over a complete noetherian local ring $R$ with residue field $k$.  Let $G_u$ be the unipotent radical of $G$ and $G_r = G/G_u$ its reductive quotient.  Let $\lambda \co \bG_m \to G$ be a one-parameter subgroup.

\begin{definition}  
  We say that $G$ is \emph{positively graded with respect to $\lambda$} if the exact sequence $1 \to G_u \to G \to G_r \to 1$ splits, the conjugation action on $\Gamma(G_u, \oh_{G_u})$  is $\lambda$-positive, and $\lambda$ is central in $G_r$.
\end{definition}

\begin{remark}
  If $G$ is defined over a field, then the conjugation action on $\Gamma(G_u, \oh_{G_u})$  is $\lambda$-positive if and only if the conjugation action of $\lambda$ on $\Lie(G_u)$ is positive, and $\lambda$ is central in $G_r$ if and only if the adjoint action of $\lambda$ on $\Lie(G_r) = \Lie(G)/\Lie(G_u)$ is trivial.  
  
  Since the $\lambda$-action on $\Gamma(G, \oh_{G_u})$ is positive and $\lambda$ is central in $G$, the parabolic $P_{\lambda}^G$ must be all of $G$, and  the short exact sequence $1 \to G_u \to G \to G_r \to 1$ is identified with the sequence $1 \to U_{\lambda} \to P_{\lambda} \to Z_{\lambda} \to 1$. 
\end{remark}

We can now prove \Cref{thm:positive-graded}:  if $G$ is positively graded with respect to $\lambda$, then
\begin{enumerate}[(1)] 
  \item $BG$ is cohomologically proper over $k$, and $(BG, BG_{k})$ is coherently complete and satisfies formal functions, and 
  \item if $G$ acts on an affine scheme $X = \Spec A$ of finite type over $R$ such that $A^G = R$, $\lambda$ acts semipositively on $A$ and $[X/G_r]$ satisfies \Cref{conj:main}, then $[X/G]$ is cohomologically proper over $A^G$ and $[X/G]$ satisfies formal functions and is coherently complete along its unique closed point.
\end{enumerate}

\begin{proof}[Proof of \Cref{thm:positive-graded}]
  Let $G_r \to G$ be a group homomorphism splitting the surjection $G \to G_r$. 
  For \eqref{thm:positive-graded1}, we consider the the morphism 
  $$BG_r \to BG$$
  induced from the splitting.  The map $BG_r \to BG$ is faithfully flat and there is an isomorphism 
  $$(BG_r/BG)^n \cong [G_u^{n-1}/G_r]$$
  where $G_r$ acts diagonally via the conjugation action. We know that $\lambda$ acts positively on $G_u = U_{\lambda}$ with $\Gamma(G_u)^{\lambda}=R$ (\Cref{ex:lambda-unipotents-are-positively-graded}), and that the same holds for the $\lambda$-action on the higher fiber products (\Cref{lem:positivity-under-products}). In particular, $(BG_r/BG)^n \to \Spec R$ are good moduli spaces.  Since $BG_r$ is cohomologically proper (\Cref{thm:main}\eqref{thm:main2}) and $(BG_r, B (G_r)_k)$ is coherently complete and satisfies formal functions, \Cref{thm:descent} implies that the same holds for $(BG, BG_k)$.

  For \eqref{thm:positive-graded2}, we first claim that $[\Spec A/G]$ has a unique closed point $x$.  We will use essentially the same argument as in the case that $G$ is linearly reductive.  Let $\pi \co [\Spec A/G] \to \Spec A^G$.  It suffices to show that if $\cZ_1, \cZ_2 \subset [\Spec A/G]$ are closed substacks, then each $\pi(\cZ_i) \subset \Spec A^G$ is closed and $\pi(\cZ_1) \cap \pi(\cZ_2) = \pi(\cZ_1 \cap \cZ_2)$.  Let $I_1$, $I_2 \subset A$ be the $G$-invariant ideals defining $\cZ_1$ and $\cZ_2$. Since $A$ is $\lambda$-positive, we have $A^G = A^{G_r}$. Thus, the image of $\cZ_i$ is identified with the image of the closed $G_r$-invariant ideal $V(I_i)$ under $[\Spec A/G_r] \to \Spec A^{G_r}$, and it follows that $\pi(\cZ_i)$ is closed by the reductive case.  Similarly, since since $I_1$, $I_2$, and $I_1 + I_2$ are $\lambda$-positive, we have  
  \small\[
   {\pi(\cZ_1 \cap \cZ_2) = V((I_1+I_2)^G) = V((I_1 + I_2)^{G_r}) = V(I_1^{G_r}) \cap V(I_2^{G_r}) = V(I_1^{G}) \cap V(I_2^{G}) = \pi(\cZ_1) \cap \pi(\cZ_2),}
 \]
 \normalsize
  where the middle equality used the reductive case.  The claim follows.
  
  We now consider the morphism $[\Spec A/G_r] \to [\Spec A / G]$ and use the following identification of the higher fiber products
  \[
    ([\Spec A/G_r]/[\Spec A/G])^n = [(\Spec A \times G_u^{n-1}) / G_r].
  \]
  Since $\lambda$ acts positively on $G_u$ with $\Gamma(G_u)^{\lambda}=R=A^G$ and $\lambda$ acts semipositively on $A$, the group of global sections of each higher fiber product is identified with $R$ by \Cref{lem:positivity-under-products}.  As $G_r$ is reductive, each higher fiber product is an adequate moduli space over $\Spec R$.  It follows from \Cref{thm:descent}, \Cref{conj:main} and \Cref{cor:reductive-ff} that $[\Spec A/G]$ is cohomologically proper, and is coherently complete and satisfies formal functions with respect to $\cG_x$.
\end{proof}
\bibliography{bibliography_coherent_completeness} 
\bibliographystyle{alpha} 
\end{document}